\documentclass{amsart}

\usepackage{Preamble}

\usepackage[
backend=biber,
style=alphabetic,
sorting=nyt,
maxbibnames=50
]{biblatex}

\title{The 2-Twist Spun Trefoil has Crossing Number Six}
\author{Sherry Gong}
\address{Texas A\&M University}
\email{sgongli@tamu.edu}

\author{Samuel Lewis-Monkman}
\address{University of Nebraska-Lincoln}
\email{smonkman2@huskers.unl.edu}

\author{Jesse Osnes}
\address{Kansas State University}
\email{osnes@ksu.edu}

\begin{document}

\maketitle

\begin{abstract}
We study the tri-plane crossing number, that is, the minimal number of crossings in a tri-plane diagram for a bridge trisection of a knotted sphere in $S^4$. We show that every 2-knot in $S^4$ that admits a bridge trisection with at most five crossings is ribbon. As a consequence, we show that the 2-twist spin of the trefoil has crossing number 6. This is the first such computation for a non-trivial knotted surface.
\end{abstract}

\section{Introduction}

In this paper, we study the complexity of bridge trisections of knotted surfaces in $S^4$. A bridge trisection of a surface link $L$ in $S^4$, originally introduced by Meier and Zupan in \cite{mz}, is a decomposition of $(S^4, L)$ into three pieces, $(B_1, D_1)$, $(B_2, D_2)$, $(B_3, D_3)$, where $B_i$ is a $4$-ball, and $D_i$ is a surface that consists of some number of disks, so that the boundary of $D_i$  is an unlink in the $S^3$ boundary of $B_i$. The bridge trisection can be thought of as a 4-dimensional analogue of the bridge splitting, which decomposes a link as two trivial tangles. 

For a bridge trisection, the intersections of the three pieces with each other determine three tangles in $B^3$, which together form a tri-plane diagram for the trisection. This tri-plane diagram completely determines the surface link $L$ up to isotopy. In this paper, we study a natural diagrammatic measure of complexity for a surface link $L$, the tri-plane crossing number, which is the minimal number of crossings in a tri-plane diagram for a bridge trisection of $L$.

The project of classifying knotted surfaces by complexity is not new. In \cite{yosh}, Yoshikawa presents a table of several surfaces in $\R^4$, organized by \textit{ch}-index, a count of crossings and hyperbolic points coming from his \textit{ch}-diagrams. This table is updated for bridge trisections in \cite{wolf}, although only upper bounds are given for crossing number. More is known for specific cases: \cite{fusion} extends Yoshikawa's table for ribbon spheres formed attaching a single tube to a pair of unlinked spheres.

Bridge trisections specifically are a relatively new object of study, and many of the canonical examples are due to \cite{wolf}. Notably, we are concerned here only with knotted spheres in $S^4$. Unknotted orientable surfaces are known to have crossing number 0 \cite{mz}, and unknotted non-orientable surfaces may require arbitrarily many crossings (depending on the number and sign of $\R P^2$-summands) \cite{wolf}.

Historically, tri-plane crossing number has been challenging to calculate, and especially to bound below. As in the classical setting, much of the work on bridge trisections is stratified by bridge number. For a classical knot, however, there is a straightforward relationship between crossing number and bridge number, so that diagrams with $n$ crossings can be enumerated with a bound on bridge number. No such relationship is known for bridge trisections, so it is not possible in general to rule out diagrams for knotted surfaces with few crossings but very high bridge number. The main result of this paper provides the first known lower bound on crossing number for bridge trisections.

\begin{thm:five-crossing-ribbon}
    Every 2-knot embedded in $S^4$ admitting a bridge trisection with at most five crossings is ribbon.
\end{thm:five-crossing-ribbon}

We show this by studying the number of Reidemeister III moves required to simplify the unlink diagrams involved in tri-plane diagrams with a certain number of crossings and using some results of Satoh that study the triple point numbers of 2-knots (\cite{S03}, \cite{S05}), as well as a result of Joseph, Meier, Miller, and Zupan that relates the triple point number of a 2-knot to the Reidemeister moves required to simplify a corresponding tri-plane diagram (\cite{jmmz}).

The restriction of low-crossing knots being ribbon, together with a diagram due to \cite{diagram} gives us the first computation of a bridge-trisected 2-knot's crossing number.

\begin{cor:2-twist}
The 2-twist spin of the trefoil has crossing number 6.
\end{cor:2-twist}

\begin{figure}[h!]
    \centering

\tikzset{every picture/.style={line width=0.75pt}} %set default line width to 0.75pt        

\begin{tikzpicture}[x=0.75pt,y=0.75pt,yscale=-1,xscale=1]
%uncomment if require: \path (0,247); %set diagram left start at 0, and has height of 247

%Straight Lines [id:da2810388578826776] 
\draw    (20,80) -- (110,80) ;
%Curve Lines [id:da22081540287467383] 
\draw [color={rgb, 255:red, 255; green, 0; blue, 0 }  ,draw opacity=1 ]   (90,80) .. controls (90.09,73.22) and (94.95,71.15) .. (97.86,73.8) .. controls (99.14,74.96) and (100.04,77.03) .. (100,80) ;
%Curve Lines [id:da4293773260994962] 
\draw [color={rgb, 255:red, 255; green, 0; blue, 0 }  ,draw opacity=1 ]   (50,80) .. controls (50.13,70.25) and (60.13,70.25) .. (60,80) ;
%Curve Lines [id:da8308921464950191] 
\draw [color={rgb, 255:red, 255; green, 0; blue, 0 }  ,draw opacity=1 ]   (30,80) .. controls (30.13,40) and (80.13,40) .. (80,80) ;
%Curve Lines [id:da6685643282137448] 
\draw [color={rgb, 255:red, 255; green, 0; blue, 0 }  ,draw opacity=1 ]   (40,80) .. controls (40,56) and (70.5,55.5) .. (70,80) ;
%Straight Lines [id:da857274728310966] 
\draw    (130,80) -- (220,80) ;
%Curve Lines [id:da11617904059769768] 
\draw [color={rgb, 255:red, 0; green, 0; blue, 255 }  ,draw opacity=1 ]   (140,80) .. controls (140.13,70.25) and (150.13,70.25) .. (150,80) ;
%Curve Lines [id:da1994461965609211] 
\draw [color={rgb, 255:red, 0; green, 0; blue, 255 }  ,draw opacity=1 ]   (179.99,80) .. controls (180.12,70.25) and (190.12,70.25) .. (189.99,80) ;
%Curve Lines [id:da5969644442115014] 
\draw [color={rgb, 255:red, 0; green, 0; blue, 255 }  ,draw opacity=1 ]   (170,80) .. controls (170,56) and (200.5,55.5) .. (200,80) ;
%Straight Lines [id:da9628030596331352] 
\draw    (80,160.5) -- (170,160.5) ;
%Curve Lines [id:da09062507861050284] 
\draw [color={rgb, 255:red, 0; green, 0; blue, 255 }  ,draw opacity=1 ]   (160,80) .. controls (160.13,40) and (210.13,40) .. (210,80) ;
%Curve Lines [id:da5416297746174711] 
\draw [color={rgb, 255:red, 126; green, 211; blue, 33 }  ,draw opacity=1 ]   (90,160) .. controls (90.22,140.33) and (90.11,120.05) .. (100,109.95) .. controls (109.89,99.84) and (118,100.37) .. (123.16,104.3) ;
%Curve Lines [id:da06482223242836205] 
\draw [color={rgb, 255:red, 126; green, 211; blue, 33 }  ,draw opacity=1 ]   (160,160) .. controls (160.22,140.11) and (159.89,119.95) .. (150,110) .. controls (140.11,100.05) and (131.22,98.89) .. (126,104) ;
%Curve Lines [id:da5853682397431738] 
\draw [color={rgb, 255:red, 126; green, 211; blue, 33 }  ,draw opacity=1 ]   (123.54,138.63) .. controls (120.79,135.13) and (119.88,125.21) .. (119.88,122.88) .. controls (119.88,120.54) and (119.79,110.96) .. (123.38,107.13) ;
%Curve Lines [id:da8469313690890307] 
\draw [color={rgb, 255:red, 126; green, 211; blue, 33 }  ,draw opacity=1 ]   (130,130) .. controls (130.64,126.57) and (130.91,121.51) .. (130.13,117.79) ;
%Curve Lines [id:da3975685311689918] 
\draw [color={rgb, 255:red, 126; green, 211; blue, 33 }  ,draw opacity=1 ]   (120,160) .. controls (120.15,155) and (119.04,144.79) .. (123.71,140.54) .. controls (128.38,136.29) and (129.13,133.46) .. (129.46,132.63) ;
%Curve Lines [id:da7417012669204899] 
\draw [color={rgb, 255:red, 126; green, 211; blue, 33 }  ,draw opacity=1 ]   (130,159.96) .. controls (130.15,154.96) and (131.29,145.07) .. (125.93,140.57) ;
%Curve Lines [id:da5887401980547481] 
\draw [color={rgb, 255:red, 126; green, 211; blue, 33 }  ,draw opacity=1 ]   (122,114.89) .. controls (129.68,110.47) and (150.21,125.31) .. (150,160.11) ;
%Curve Lines [id:da08493115414434205] 
\draw [color={rgb, 255:red, 126; green, 211; blue, 33 }  ,draw opacity=1 ]   (122.27,130.46) .. controls (127.16,127.11) and (140.1,134.88) .. (139.89,160) ;
%Curve Lines [id:da9675246088184256] 
\draw [color={rgb, 255:red, 126; green, 211; blue, 33 }  ,draw opacity=1 ]   (119.47,131.84) .. controls (119.16,132.05) and (110.21,135) .. (110,160) ;
%Curve Lines [id:da2527969169058081] 
\draw [color={rgb, 255:red, 126; green, 211; blue, 33 }  ,draw opacity=1 ]   (118.32,116.58) .. controls (116.84,117.11) and (100,125) .. (100,160) ;
%Curve Lines [id:da3028568043078126] 
\draw [color={rgb, 255:red, 126; green, 211; blue, 33 }  ,draw opacity=1 ]   (129.68,113.53) .. controls (128.74,110.05) and (125.37,105.84) .. (123.16,104.3) ;

% Text Node
\draw (65,90) node  [font=\footnotesize]  {$T_{\alpha }$};
% Text Node
\draw (175,90) node  [font=\footnotesize]  {$T_{\beta }$};
% Text Node
\draw (125,170.5) node  [font=\footnotesize]  {$T_{\gamma }$};

\end{tikzpicture}
    \caption{A 6-crossing tri-plane diagram of the 2-twist spun trefoil knot}
    \label{fig:tri-plane}
\end{figure}
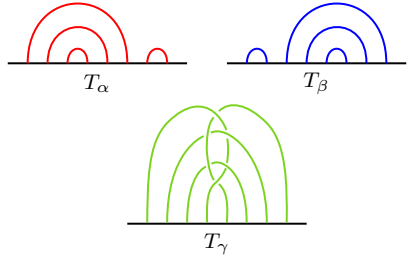

In Section 2, we establish some necessary background on bridge trisections, tri-plane diagrams, and classical diagrams of unlinks. We also discuss the notion of ribbonness for a knotted sphere and its relation to the singularity set of a surface in $S^4$ immersed into an equatorial $S^3$. Section 3 establishes two lemmas regarding low-crossing unlink diagrams and presents the main theorem and some interesting further questions.

\vspace{0.2 cm}

\subsection*{Acknowledgments}

We would like to thank Alex Zupan and Maggie Miller for their correspondence and guidance throughout this project, and especially Jeffrey Meier both for his excellent advice and for providing the six-crossing diagram of the 2-twist spun trefoil.

SG was partially supported by NSF Award 2340465. The authors would like to thank the Trisectors Workshop 2024: Connections with Knotted Surfaces, where this project originated.
\section{Preliminaries}

Surface knots and links were studied via trisections for the first time by Meier and Zupan in \cite{mz}, following the work of Gay and Kirby on trisections of 4-manifolds \cite{gk}.

\begin{defn}[Bridge Trisection]
    A $(b;c_1,c_2,c_3)$ \textit{bridge trisection} $\mathcal{T}$ of a surface link $L$ in a 4-manifold $X$ is a decomposition $(X,L) = (X_1,D_1) \cup (X_2,D_2) \cup (X_3,D_3)$ with each $X_i$ a 4-dimensional handlebody such that 
    \begin{itemize}
        \item $X_1 \cup X_2 \cup X_3$ is a trisection of $X$,
        \item $D_i$ is a trivial $c_i$-disk system in $X_i$ with $\del D_i \subset \del X_i$, and
        \item Each pairwise intersection $(X_i,D_i) \cap (X_j,D_j)$ (with $i \ne j$) is a trivial tangle.
    \end{itemize}
\end{defn}

Note that a trivial system of disks is a collection of disks properly embedded in the relevant manifold with boundary such that all disks are simultaneously isotopic into the boundary.

We will call the triple $(T_\a,T_\b,T_\g) := (D_1 \cap D_2, D_2 \cap D_3, D_3 \cap D_1)$ for a bridge trisection $\mathcal{T}$ of surface link $L$ in $S^4$ (where the trisection of $S^4$ is the standard genus zero one) a \textit{tri-plane diagram} for $L$. Every surface link in $S^4$ admits a bridge trisection and tri-plane diagram, and these are unique up to certain moves outlined in \cite{mz}.

We see immediately from this definition that each pair $\del D_i = T_j \cup \overline{T_k}$ is an unlink inheriting a $b$-bridge splitting from the trisection. 

We will need to be very careful with our treatment of link diagrams, so we will give these a formal definition, calling attention to the fact that a link diagram is in fact a graph.

\begin{defn}
    A \textit{tangle diagram} (or, in the case of a closed tangle, \textit{knot or link diagram}) is a 4-valent planar graph where each vertex contains ``crossing information:" two edges incident to the vertex are broken to indicate an ``understrand." The unbroken edges together form an ``overstrand."
\end{defn}

Note that, by removing crossing information from a link diagram, we obtain a graph representing a number of different diagrams.

The most natural stratification of link diagrams and tangles is by the number of crossings they contain, and links are most commonly organized by the minimum of this value over all diagrams, called the \textit{crossing number} of a link. This is one of the best-studied invariants of classical links. It has a natural analogue for surface links and tri-plane diagrams.

\begin{defn}[Tri-Plane Crossing Number]
    Let $L$ be a surface link in $S^4$ described by a tri-plane diagram $\mathcal{D}= (T_\a,T_\b,T_\g)$. The \textit{crossing number of} $\mathcal{D}$ is $c(\mathcal{D}) := c(T_\a)+c(T_\b)+c(T_\g)$, where $c(T_i)$ is the crossing number of the tangle $T_i$. The \textit{crossing number} of $L$ is $\min\{c(\mathcal{D}):\mathcal{D}\text{ is a tri-plane diagram of L}\}$.
\end{defn}

Crossing number is invariant by construction, but is hard to calculate for a given surface. To the authors' knowledge, there are currently no known crossing numbers for non-trivial knotted surfaces. In the next section, we present the first (Cor \ref{cor:2-twist}).

To study crossing numbers of surface knots (and links), we will need to make careful consideration of non-standard unlink diagrams.
Recall a classical theorem of Reidemeister, that any two diagrams of a classical link $L$ (in, for example, $S^3$) are related by planar isotopy and a finite sequence of moves from Figure \ref{fig:reidemeister_moves}, which we will hereafter refer to as RI-, RII-, and RIII-moves. In particular, any diagram of a $c$-component unlink can be made standard (that is, made to look like $c$ circles embedded in the plane with no crossings) via such a sequence. We will call such a sequence of Reidemeister moves for an unlink diagram an \textit{untying} of the diagram.

\begin{figure}[h]
    \centering

\tikzset{every picture/.style={line width=0.75pt}} %set default line width to 0.75pt        

\begin{tikzpicture}[x=0.75pt,y=0.75pt,yscale=-1,xscale=1]
%uncomment if require: \path (0,232); %set diagram left start at 0, and has height of 232

%Curve Lines [id:da6509166836849047] 
\draw [line width=1.5]    (36.66,64.55) .. controls (28.67,58.63) and (20,51.37) .. (20,39.09) .. controls (20,26.81) and (32.08,19.92) .. (45,20) .. controls (57.91,20.08) and (70.33,26.68) .. (69.99,39.09) .. controls (69.66,51.5) and (61.66,58.37) .. (53.33,64.55) ;
%Straight Lines [id:da668175568191172] 
\draw [line width=1.5]    (53.33,77.27) -- (69.99,90) ;
%Straight Lines [id:da23052203497493828] 
\draw [line width=1.5]    (53.33,64.55) -- (20,90) ;
%Curve Lines [id:da4399319985993497] 
\draw [line width=1.5]    (170,70) .. controls (201,70.13) and (179.96,39.97) .. (200,40) .. controls (220.04,40.03) and (199.25,69.88) .. (230,70) ;
%Straight Lines [id:da8202820927136859] 
\draw    (93,51) -- (147,51) ;
\draw [shift={(150,51)}, rotate = 180] [fill={rgb, 255:red, 0; green, 0; blue, 0 }  ][line width=0.08]  [draw opacity=0] (8.93,-4.29) -- (0,0) -- (8.93,4.29) -- cycle    ;
\draw [shift={(90,51)}, rotate = 0] [fill={rgb, 255:red, 0; green, 0; blue, 0 }  ][line width=0.08]  [draw opacity=0] (8.93,-4.29) -- (0,0) -- (8.93,4.29) -- cycle    ;
%Curve Lines [id:da7380165976666576] 
\draw [line width=1.5]    (330,70) .. controls (320.29,60.57) and (320.57,50.86) .. (330,40) ;
%Curve Lines [id:da22574809539925744] 
\draw [line width=1.5]    (320,90) .. controls (350,69.78) and (350.22,40) .. (320,20) ;
%Straight Lines [id:da8825097451752648] 
\draw [line width=1.5]    (340,30) -- (350,20) ;
%Straight Lines [id:da4570441422955934] 
\draw [line width=1.5]    (340,80) -- (350,90) ;
%Straight Lines [id:da5021281294230975] 
\draw    (373,50) -- (437,50) ;
\draw [shift={(440,50)}, rotate = 180] [fill={rgb, 255:red, 0; green, 0; blue, 0 }  ][line width=0.08]  [draw opacity=0] (8.93,-4.29) -- (0,0) -- (8.93,4.29) -- cycle    ;
\draw [shift={(370,50)}, rotate = 0] [fill={rgb, 255:red, 0; green, 0; blue, 0 }  ][line width=0.08]  [draw opacity=0] (8.93,-4.29) -- (0,0) -- (8.93,4.29) -- cycle    ;
%Curve Lines [id:da5547133263864538] 
\draw [line width=1.5]    (460,90) .. controls (470.22,69.56) and (469.78,40.22) .. (460,20) ;
%Curve Lines [id:da2137688355868952] 
\draw [line width=1.5]    (490,90) .. controls (480.89,70.22) and (480,40) .. (490,20) ;
%Straight Lines [id:da46082568953523184] 
\draw [line width=1.5]    (180,180) -- (230,130) ;
%Straight Lines [id:da32428182571479436] 
\draw [line width=1.5]    (160,130) -- (170,140) ;
%Straight Lines [id:da4652361118644467] 
\draw [line width=1.5]    (200,170) -- (230,200) ;
%Straight Lines [id:da9091172139431306] 
\draw [line width=1.5]    (320,200) -- (370,150) ;
%Straight Lines [id:da4984084991157388] 
\draw [line width=1.5]    (320,130) -- (350,160) ;
%Straight Lines [id:da6331449994938027] 
\draw [line width=1.5]    (360,170) -- (370,180) ;
%Straight Lines [id:da9889839618273737] 
\draw    (243,160) -- (307,160) ;
\draw [shift={(310,160)}, rotate = 180] [fill={rgb, 255:red, 0; green, 0; blue, 0 }  ][line width=0.08]  [draw opacity=0] (8.93,-4.29) -- (0,0) -- (8.93,4.29) -- cycle    ;
\draw [shift={(240,160)}, rotate = 0] [fill={rgb, 255:red, 0; green, 0; blue, 0 }  ][line width=0.08]  [draw opacity=0] (8.93,-4.29) -- (0,0) -- (8.93,4.29) -- cycle    ;
%Straight Lines [id:da7835259056317385] 
\draw [line width=1.5]    (160,200) -- (170,190) ;
%Straight Lines [id:da5678192550038498] 
\draw [line width=1.5]    (180,150) -- (190,160) ;
%Curve Lines [id:da8087594069531903] 
\draw [line width=1.5]    (180,200) .. controls (170.89,180.22) and (170,150) .. (180,130) ;
%Curve Lines [id:da14874171675069214] 
\draw [line width=1.5]    (370,200) .. controls (380.22,179.56) and (379.78,150.22) .. (370,130) ;
%Straight Lines [id:da2888608105905427] 
\draw [line width=1.5]    (380,190) -- (390,200) ;
%Straight Lines [id:da9244523417661055] 
\draw [line width=1.5]    (380,140) -- (390,130) ;

% Text Node
\draw (120.5,60.5) node  [font=\footnotesize] [align=left] {\begin{minipage}[lt]{26.95pt}\setlength\topsep{0pt}
Type I
\end{minipage}};
% Text Node
\draw (406,60.5) node  [font=\footnotesize] [align=left] {\begin{minipage}[lt]{33.53pt}\setlength\topsep{0pt}
 Type II
\end{minipage}};
% Text Node
\draw (276,170.5) node  [font=\footnotesize] [align=left] {\begin{minipage}[lt]{33.53pt}\setlength\topsep{0pt}
Type III
\end{minipage}};

\end{tikzpicture}
    \caption{Reidemeister Moves}
    \label{fig:reidemeister_moves}
\end{figure}
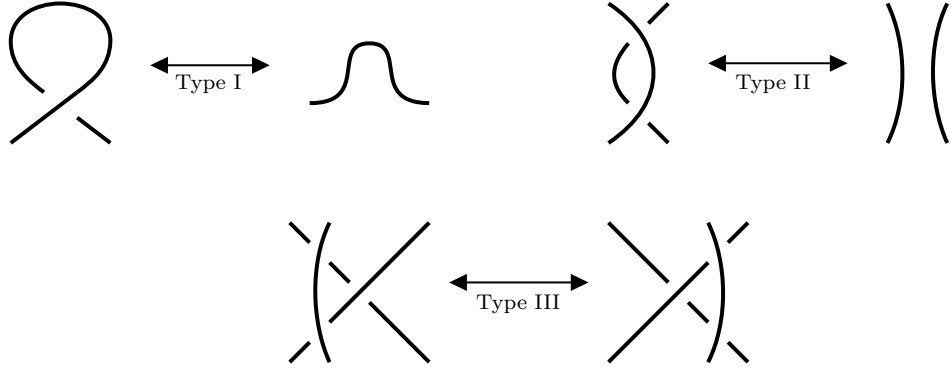

An important feature of some 2-knots we will discuss here is \textit{ribbonness}.

\begin{defn}

A 2-knot is \textit{ribbon} if it is obtained by adding $k$ 1-handles to a trivial $(k+1)$ component 2-link. That is, a ribbon 2-knot is obtained by attaching ``thin tubes," i.e. copies of $S^1 \x I$ viewed as the boundary of a tubular neighborhood of some arc $\a$, to unknotted $S^2$'s by, for each boundary component $S^1 \x \{0\}$ or $S^1 \x \{1\}$, removing a trivial disk in the relevant $S^2$ and identifying the boundary of the puncture with the attaching circle.
    
\end{defn}

A projection of a ribbon 2-knot obtained by adding a single 1-handle to a trivial $2$ component 2-link is depicted in Figure \ref{fig:ribbon_knot_projection}. 

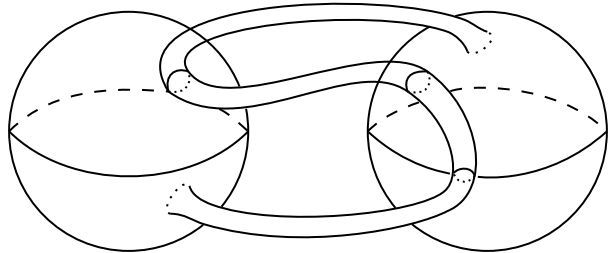
\begin{figure}[ht!]
\centering

\tikzset{every picture/.style={line width=0.75pt}} %set default line width to 0.75pt        

\begin{tikzpicture}[x=0.75pt,y=0.75pt,yscale=-1,xscale=1]
%uncomment if require: \path (0,161); %set diagram left start at 0, and has height of 161

%Shape: Arc [id:dp02144571199392642] 
\draw  [draw opacity=0] (127.38,126.94) .. controls (117.12,135.11) and (104.13,140) .. (90,140) .. controls (56.86,140) and (30,113.14) .. (30,80) .. controls (30,46.86) and (56.86,20) .. (90,20) .. controls (115.15,20) and (136.69,35.48) .. (145.61,57.43) -- (90,80) -- cycle ; \draw   (127.38,126.94) .. controls (117.12,135.11) and (104.13,140) .. (90,140) .. controls (56.86,140) and (30,113.14) .. (30,80) .. controls (30,46.86) and (56.86,20) .. (90,20) .. controls (115.15,20) and (136.69,35.48) .. (145.61,57.43) ;  
%Curve Lines [id:da4748784210022726] 
\draw    (30,80) .. controls (60.27,110.18) and (120.27,110) .. (150,80) ;
%Curve Lines [id:da9538398841839696] 
\draw  [dash pattern={on 4.5pt off 4.5pt}]  (30,80) .. controls (53.43,56.14) and (89.14,58.43) .. (110,60) ;
%Shape: Arc [id:dp1270146238297748] 
\draw  [draw opacity=0] (253.89,22.19) .. controls (259.02,20.76) and (264.42,20) .. (270,20) .. controls (303.14,20) and (330,46.86) .. (330,80) .. controls (330,113.14) and (303.14,140) .. (270,140) .. controls (255.71,140) and (242.58,135) .. (232.27,126.66) -- (270,80) -- cycle ; \draw   (253.89,22.19) .. controls (259.02,20.76) and (264.42,20) .. (270,20) .. controls (303.14,20) and (330,46.86) .. (330,80) .. controls (330,113.14) and (303.14,140) .. (270,140) .. controls (255.71,140) and (242.58,135) .. (232.27,126.66) ;  
%Curve Lines [id:da5762121566752869] 
\draw    (210,80) .. controls (229.44,98.11) and (249,100.56) .. (251.22,101.22) ;
%Curve Lines [id:da13684145767278266] 
\draw  [dash pattern={on 4.5pt off 4.5pt}]  (251.44,59.89) .. controls (258.11,56.78) and (304.33,54.78) .. (330,80) ;
%Curve Lines [id:da5234708335319881] 
\draw    (110.83,121.33) .. controls (118.33,121.33) and (121.89,124.89) .. (128.07,126.68) .. controls (177.05,140.85) and (233.57,128.58) .. (250,120) .. controls (273.17,110.33) and (267.17,69.33) .. (240,50) .. controls (212.83,30.67) and (138.83,75.67) .. (120,50) .. controls (102.96,26.77) and (212.06,16.65) .. (250,30) .. controls (253.99,31.4) and (258.83,36) .. (260,40) ;
%Curve Lines [id:da5302091357330941] 
\draw    (120.5,107.33) .. controls (121.67,112) and (126.75,114.95) .. (131.68,116.75) .. controls (164.21,128.63) and (235.33,121.58) .. (248.5,110) .. controls (258.17,102) and (251.5,74.33) .. (230,60) .. controls (202.83,40.67) and (146.5,86.67) .. (110,60) .. controls (80.17,13) and (216.83,8.42) .. (258.5,24) .. controls (262.6,25.53) and (266.17,29) .. (270,30) ;
%Curve Lines [id:da9901118305687185] 
\draw  [dash pattern={on 0.84pt off 2.51pt}]  (110.83,121.33) .. controls (104.83,119.67) and (116.5,102.67) .. (120.5,107.33) ;
%Curve Lines [id:da5360030835620256] 
\draw  [dash pattern={on 0.84pt off 2.51pt}]  (260,40) .. controls (263.33,44) and (277.67,33.67) .. (270,30) ;
%Curve Lines [id:da587489671663693] 
\draw  [dash pattern={on 0.75pt off 1.5pt}]  (263.67,102.33) .. controls (260.78,106.56) and (254.33,106.33) .. (253.17,100.89) ;
%Curve Lines [id:da3746671301819151] 
\draw    (263.67,102.33) .. controls (262.33,97.67) and (255.67,97.89) .. (253.17,100.89) ;
%Curve Lines [id:da17662992946047562] 
\draw  [dash pattern={on 0.75pt off 1.5pt}]  (230,60) .. controls (239.25,62.25) and (243,56.75) .. (240,50) ;
%Curve Lines [id:da6796747557323852] 
\draw    (230,60) .. controls (227.75,53.75) and (231,49.25) .. (240,50) ;
%Curve Lines [id:da36111003590734136] 
\draw  [dash pattern={on 0.75pt off 1.5pt}]  (110,60) .. controls (119,61.5) and (121.75,55.75) .. (120,50) ;
%Curve Lines [id:da581206198306846] 
\draw    (110,60) .. controls (107.75,52) and (113.25,47.5) .. (120,50) ;
%Shape: Arc [id:dp08758063599012444] 
\draw  [draw opacity=0] (223.91,118.42) .. controls (215.22,108.01) and (210,94.61) .. (210,80) .. controls (210,57.31) and (222.59,37.56) .. (241.17,27.37) -- (270,80) -- cycle ; \draw   (223.91,118.42) .. controls (215.22,108.01) and (210,94.61) .. (210,80) .. controls (210,57.31) and (222.59,37.56) .. (241.17,27.37) ;  
%Curve Lines [id:da3730159363517793] 
\draw    (265.22,102.78) .. controls (266.78,103.11) and (300,107) .. (330,80) ;
%Curve Lines [id:da3686219944430248] 
\draw  [dash pattern={on 4.5pt off 4.5pt}]  (210,80) .. controls (217,72.78) and (227.67,68.11) .. (235.67,65.44) ;
%Shape: Arc [id:dp4262782788169547] 
\draw  [draw opacity=0] (148.93,68.66) .. controls (149.63,72.33) and (150,76.12) .. (150,80) .. controls (150,94.34) and (144.97,107.51) .. (136.57,117.83) -- (90,80) -- cycle ; \draw   (148.93,68.66) .. controls (149.63,72.33) and (150,76.12) .. (150,80) .. controls (150,94.34) and (144.97,107.51) .. (136.57,117.83) ;  
%Curve Lines [id:da4181275649622035] 
\draw  [dash pattern={on 4.5pt off 4.5pt}]  (135.43,68.71) .. controls (142.57,72.43) and (145.43,76.14) .. (150,80) ;

\end{tikzpicture}
\caption{\label{fig:ribbon_knot_projection} A projection of a ribbon 2-knot with two trivial spheres and one 1-handle.}
\end{figure}

The most relevant fact about ribbon surfaces is that a 2-knot in $S^4$ is ribbon if and only if it has an immersion into an equatorial $S^3$ without any triple points (that is, without any points in the projection with preimage set of size three) \cite{S05}. To see this, refer first to \cite{yaj} (where a ``symmetric ribbon sphere" is in our language a ribbon 2-knot and a ``simply knotted sphere" is a 2-knot with an immersion into $S^3$ without triple or branch points) and then deal with branch points in the following way.

\begin{rmk}
    A quick note on notation: Throughout this paper, we will be using ``immersion'' to be what Satoh, in his papers, calls a ``diagram", that is, a projection into $S^3$ with only double points, triple points, and branch points as its singular points. In particular, we do \textit{not} require them \textit{a priori} to be immersions in the traditional sense of the differential being injective.
\end{rmk}
   
Suppose that a 2-knot $K$ is immersed in $S^3$ such that it contains no triple points of self-intersection. Then the set of all self-intersections consists of a collection of simple closed curves and arcs. Each curve consists of double points, and each arc's interior consists of double points, while each arc's endpoints are branch points. Notably, a given arc is disjoint from all others and from all curves. If an arc $\a$ appears as in Figure \ref{fig:Rosec}, it can be eliminated via a Roseman move (of type c) \cite{rose} as shown. In any other case, push off the interior of $\a$ to find an arc $\a'$ sharing endpoints with $\a$. Then $\a'$ traces out the path of a Roseman move of type d, as in Figure \ref{fig:Rosed}. Thus, any immersion of a 2-knot $K$ without triple points is isotopic to one without branch points, and thus $K$ is ribbon.

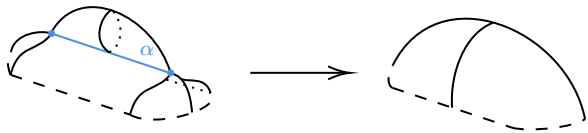
\begin{figure}[h!]
    \centering

\tikzset{every picture/.style={line width=0.75pt}} %set default line width to 0.75pt        

\begin{tikzpicture}[x=0.75pt,y=0.75pt,yscale=-1,xscale=1]
%uncomment if require: \path (0,185); %set diagram left start at 0, and has height of 185

%Straight Lines [id:da4056926838440925] 
\draw [color={rgb, 255:red, 74; green, 144; blue, 226 }  ,draw opacity=1 ][line width=0.75]    (90.25,90.25) -- (150.25,110.25) ;
%Curve Lines [id:da9821779136664586] 
\draw    (90,90) .. controls (104.33,63.67) and (141.33,81) .. (150,110) ;
%Curve Lines [id:da48927220995365983] 
\draw    (119.22,99.67) .. controls (110.56,94.11) and (114.78,81.89) .. (120.33,78.33) ;
%Curve Lines [id:da2317646513525886] 
\draw  [dash pattern={on 0.84pt off 2.51pt}]  (119.22,99.67) .. controls (127.44,95.44) and (125.67,82.78) .. (120.33,78.33) ;
%Curve Lines [id:da2283795871598272] 
\draw    (129.67,131) .. controls (133.89,114.78) and (143.22,120.11) .. (150.25,110.25) ;
%Curve Lines [id:da0537626083120234] 
\draw    (69.67,111) .. controls (73.89,94.78) and (83.22,100.11) .. (90.25,90.25) ;
%Straight Lines [id:da2009027673684649] 
\draw  [dash pattern={on 4.5pt off 4.5pt}]  (69.67,111) -- (129.67,131) ;
%Curve Lines [id:da5454333335196146] 
\draw    (155,113) .. controls (159.67,115) and (167.22,111.44) .. (170.11,120.11) ;
%Straight Lines [id:da8182912288916743] 
\draw  [dash pattern={on 0.84pt off 2.51pt}]  (147.89,113.44) -- (170.11,120.11) ;
%Curve Lines [id:da16664559524238232] 
\draw  [dash pattern={on 4.5pt off 4.5pt}]  (129.67,131) .. controls (142.11,138.33) and (177.22,127.89) .. (170.11,120.11) ;
%Curve Lines [id:da8358244189148266] 
\draw    (160.78,129.89) .. controls (159.89,119) and (156.78,111.67) .. (150.25,110.25) ;
%Curve Lines [id:da1341122389514533] 
\draw    (69.89,99.67) .. controls (71.44,90.11) and (85,88.33) .. (90.25,90.25) ;
%Curve Lines [id:da6160282687362024] 
\draw  [dash pattern={on 4.5pt off 4.5pt}]  (69.67,111) .. controls (67.89,106.78) and (67.89,102.11) .. (69.89,99.67) ;
%Shape: Circle [id:dp6780784294564843] 
\draw  [color={rgb, 255:red, 74; green, 144; blue, 226 }  ,draw opacity=1 ][fill={rgb, 255:red, 74; green, 144; blue, 226 }  ,fill opacity=1 ] (148.75,110) .. controls (148.75,109.31) and (149.31,108.75) .. (150,108.75) .. controls (150.69,108.75) and (151.25,109.31) .. (151.25,110) .. controls (151.25,110.69) and (150.69,111.25) .. (150,111.25) .. controls (149.31,111.25) and (148.75,110.69) .. (148.75,110) -- cycle ;
%Shape: Circle [id:dp560206654581342] 
\draw  [color={rgb, 255:red, 74; green, 144; blue, 226 }  ,draw opacity=1 ][fill={rgb, 255:red, 74; green, 144; blue, 226 }  ,fill opacity=1 ] (89,90.25) .. controls (89,89.56) and (89.56,89) .. (90.25,89) .. controls (90.94,89) and (91.5,89.56) .. (91.5,90.25) .. controls (91.5,90.94) and (90.94,91.5) .. (90.25,91.5) .. controls (89.56,91.5) and (89,90.94) .. (89,90.25) -- cycle ;
%Straight Lines [id:da7884544797187594] 
\draw    (190,110) -- (238,110) ;
\draw [shift={(240,110)}, rotate = 180] [color={rgb, 255:red, 0; green, 0; blue, 0 }  ][line width=0.75]    (10.93,-3.29) .. controls (6.95,-1.4) and (3.31,-0.3) .. (0,0) .. controls (3.31,0.3) and (6.95,1.4) .. (10.93,3.29)   ;
%Curve Lines [id:da6080364482674067] 
\draw    (261.53,106.18) .. controls (284.33,60.83) and (349,87.83) .. (358,133.17) ;
%Curve Lines [id:da0015373514300318636] 
\draw    (291.31,127.51) .. controls (290,109.83) and (298,89.83) .. (311.97,84.84) ;
%Straight Lines [id:da09976783510266662] 
\draw  [dash pattern={on 4.5pt off 4.5pt}]  (261.31,117.51) -- (321.31,137.51) ;
%Curve Lines [id:da3177557507340063] 
\draw  [dash pattern={on 4.5pt off 4.5pt}]  (321.31,137.51) .. controls (329.6,139.8) and (352.8,137.4) .. (358,133.17) ;
%Curve Lines [id:da4960016469311087] 
\draw  [dash pattern={on 4.5pt off 4.5pt}]  (261.31,117.51) .. controls (258,115.83) and (259.53,108.62) .. (261.53,106.18) ;

% Text Node
\draw (133,95) node [anchor=north west][inner sep=0.75pt]  [font=\footnotesize,color={rgb, 255:red, 74; green, 144; blue, 226 }  ,opacity=1 ]  {$\alpha $};

\end{tikzpicture}
    \caption{A Roseman move of type c}
    \label{fig:Rosec}
\end{figure}

\begin{figure}[h!]
    \centering

\tikzset{every picture/.style={line width=0.75pt}} %set default line width to 0.75pt        

\begin{tikzpicture}[x=0.75pt,y=0.75pt,yscale=-1,xscale=1]
%uncomment if require: \path (0,185); %set diagram left start at 0, and has height of 185

%Curve Lines [id:da17280132834359896] 
\draw [color={rgb, 255:red, 208; green, 2; blue, 27 }  ,draw opacity=1 ]   (287.25,104.25) .. controls (291.33,108.17) and (300.33,110.83) .. (306.5,111) ;
%Curve Lines [id:da2388876792673441] 
\draw  [dash pattern={on 0.84pt off 2.51pt}]  (496.28,102.48) .. controls (498.63,105.74) and (498,109.63) .. (495.78,110.98) ;
%Straight Lines [id:da8148836448304375] 
\draw [color={rgb, 255:red, 74; green, 144; blue, 226 }  ,draw opacity=1 ][line width=0.75]  [dash pattern={on 0.84pt off 2.51pt}]  (321.17,103.67) -- (340.5,110.25) ;
%Curve Lines [id:da6095338969486277] 
\draw [color={rgb, 255:red, 208; green, 2; blue, 27 }  ,draw opacity=1 ]   (90.25,90.25) .. controls (86.14,106.14) and (119.57,119) .. (150,110) ;
%Curve Lines [id:da09459170294154395] 
\draw    (137.35,83) .. controls (144.15,92.2) and (144.8,94.05) .. (150.25,110.25) ;
%Curve Lines [id:da1889485424916323] 
\draw    (90.25,90.25) .. controls (97,85) and (98.6,78.2) .. (103.8,75.4) ;
%Curve Lines [id:da025651406500206808] 
\draw  [dash pattern={on 0.84pt off 2.51pt}]  (119.22,99.67) .. controls (125.1,96) and (127.3,88) .. (128.1,84.2) ;
%Curve Lines [id:da7740652122540935] 
\draw    (69.89,99.67) .. controls (71.44,90.11) and (85,88.33) .. (90.25,90.25) ;
%Straight Lines [id:da46434522069990813] 
\draw [color={rgb, 255:red, 74; green, 144; blue, 226 }  ,draw opacity=1 ][line width=0.75]    (90.25,90.25) -- (150.25,110.25) ;
%Curve Lines [id:da23759510732863742] 
\draw    (129.67,131) .. controls (133.89,114.78) and (143.22,120.11) .. (150.25,110.25) ;
%Curve Lines [id:da4142022283692274] 
\draw    (69.67,111) .. controls (73.89,94.78) and (83.22,100.11) .. (90.25,90.25) ;
%Straight Lines [id:da8367618652992143] 
\draw  [dash pattern={on 4.5pt off 4.5pt}]  (69.67,111) -- (129.67,131) ;
%Curve Lines [id:da6108042818659741] 
\draw    (155,113) .. controls (159.67,115) and (167.22,111.44) .. (170.11,120.11) ;
%Straight Lines [id:da02681604040503749] 
\draw  [dash pattern={on 0.84pt off 2.51pt}]  (147.89,113.44) -- (170.11,120.11) ;
%Curve Lines [id:da13196983581825672] 
\draw  [dash pattern={on 4.5pt off 4.5pt}]  (129.67,131) .. controls (142.11,138.33) and (177.22,127.89) .. (170.11,120.11) ;
%Curve Lines [id:da5542810489606698] 
\draw    (160.78,129.89) .. controls (159.89,119) and (156.78,111.67) .. (150.25,110.25) ;
%Curve Lines [id:da4527972059566566] 
\draw  [dash pattern={on 4.5pt off 4.5pt}]  (69.67,111) .. controls (67.89,106.78) and (67.89,102.11) .. (69.89,99.67) ;
%Shape: Circle [id:dp4384700383552008] 
\draw  [color={rgb, 255:red, 74; green, 144; blue, 226 }  ,draw opacity=1 ][fill={rgb, 255:red, 74; green, 144; blue, 226 }  ,fill opacity=1 ] (148.75,110) .. controls (148.75,109.31) and (149.31,108.75) .. (150,108.75) .. controls (150.69,108.75) and (151.25,109.31) .. (151.25,110) .. controls (151.25,110.69) and (150.69,111.25) .. (150,111.25) .. controls (149.31,111.25) and (148.75,110.69) .. (148.75,110) -- cycle ;
%Straight Lines [id:da9965887924736889] 
\draw    (190,110) -- (238,110) ;
\draw [shift={(240,110)}, rotate = 180] [color={rgb, 255:red, 0; green, 0; blue, 0 }  ][line width=0.75]    (10.93,-3.29) .. controls (6.95,-1.4) and (3.31,-0.3) .. (0,0) .. controls (3.31,0.3) and (6.95,1.4) .. (10.93,3.29)   ;
%Shape: Circle [id:dp8398862056524166] 
\draw  [color={rgb, 255:red, 74; green, 144; blue, 226 }  ,draw opacity=1 ][fill={rgb, 255:red, 74; green, 144; blue, 226 }  ,fill opacity=1 ] (89,90.25) .. controls (89,89.56) and (89.56,89) .. (90.25,89) .. controls (90.94,89) and (91.5,89.56) .. (91.5,90.25) .. controls (91.5,90.94) and (90.94,91.5) .. (90.25,91.5) .. controls (89.56,91.5) and (89,90.94) .. (89,90.25) -- cycle ;
%Curve Lines [id:da3084393723286414] 
\draw    (119.22,99.67) .. controls (110.56,94.11) and (112.9,85.4) .. (113.1,81) ;
%Curve Lines [id:da8777030048545639] 
\draw  [dash pattern={on 4.5pt off 4.5pt}]  (103.8,75.4) .. controls (105.3,81.2) and (131.1,86) .. (137.35,83) ;
%Curve Lines [id:da18605793599198195] 
\draw  [dash pattern={on 4.5pt off 4.5pt}]  (103.8,75.4) .. controls (111.5,70.6) and (142.7,77.6) .. (137.35,83) ;
%Curve Lines [id:da47844145074761935] 
\draw    (128.1,84.2) .. controls (128.7,82.8) and (129.1,80.8) .. (128.9,76.8) ;
%Curve Lines [id:da4879258088477032] 
\draw    (327.6,83) .. controls (330.83,87.17) and (332.17,93) .. (336.33,97.17) ;
%Curve Lines [id:da6663138473699474] 
\draw    (286.94,82.82) .. controls (291.06,79.76) and (290.59,77.18) .. (294.05,75.4) ;
%Curve Lines [id:da1620547902102576] 
\draw  [dash pattern={on 0.84pt off 2.51pt}]  (309.47,99.67) .. controls (315.35,96) and (317.55,88) .. (318.35,84.2) ;
%Curve Lines [id:da4955067634858229] 
\draw    (260.14,99.67) .. controls (261.69,90.11) and (275.25,88.33) .. (280.5,90.25) ;
%Straight Lines [id:da047146236344133396] 
\draw [color={rgb, 255:red, 74; green, 144; blue, 226 }  ,draw opacity=1 ][line width=0.75]    (280.5,90.25) -- (321.17,103.67) ;
%Curve Lines [id:da22535624519408304] 
\draw    (319.92,131) .. controls (324.14,114.78) and (333.47,120.11) .. (340.5,110.25) ;
%Curve Lines [id:da5348719497291272] 
\draw    (259.92,111) .. controls (264.14,94.78) and (273.47,100.11) .. (280.5,90.25) ;
%Straight Lines [id:da4189039003121463] 
\draw  [dash pattern={on 4.5pt off 4.5pt}]  (259.92,111) -- (319.92,131) ;
%Curve Lines [id:da4201200368064031] 
\draw    (345.25,113) .. controls (349.92,115) and (357.47,111.44) .. (360.36,120.11) ;
%Straight Lines [id:da21885581719615943] 
\draw  [dash pattern={on 0.84pt off 2.51pt}]  (338.14,113.44) -- (360.36,120.11) ;
%Curve Lines [id:da48471053709382317] 
\draw  [dash pattern={on 4.5pt off 4.5pt}]  (319.92,131) .. controls (332.36,138.33) and (367.47,127.89) .. (360.36,120.11) ;
%Curve Lines [id:da020705852101041144] 
\draw    (351.03,129.89) .. controls (350.14,119) and (347.03,111.67) .. (340.5,110.25) ;
%Curve Lines [id:da6792647672928805] 
\draw  [dash pattern={on 4.5pt off 4.5pt}]  (259.92,111) .. controls (258.14,106.78) and (258.14,102.11) .. (260.14,99.67) ;
%Curve Lines [id:da24056107364802892] 
\draw    (309.47,99.67) .. controls (300.81,94.11) and (303.15,85.4) .. (303.35,81) ;
%Curve Lines [id:da3173398223995242] 
\draw  [dash pattern={on 4.5pt off 4.5pt}]  (294.05,75.4) .. controls (295.55,81.2) and (321.35,86) .. (327.6,83) ;
%Curve Lines [id:da6791666197924998] 
\draw  [dash pattern={on 4.5pt off 4.5pt}]  (294.05,75.4) .. controls (301.75,70.6) and (332.95,77.6) .. (327.6,83) ;
%Curve Lines [id:da806135826623957] 
\draw    (318.35,84.2) .. controls (318.95,82.8) and (319.35,80.8) .. (319.15,76.8) ;
%Curve Lines [id:da22251496582685237] 
\draw [color={rgb, 255:red, 74; green, 144; blue, 226 }  ,draw opacity=1 ]   (280.5,90.25) .. controls (280.12,92.82) and (280.59,96) .. (287.25,104.25) ;
%Curve Lines [id:da8681496470015561] 
\draw [color={rgb, 255:red, 74; green, 144; blue, 226 }  ,draw opacity=1 ]   (306.5,111) .. controls (319.25,114.75) and (330.5,113.5) .. (340.5,110.25) ;
%Shape: Circle [id:dp8603009974475522] 
\draw  [color={rgb, 255:red, 74; green, 144; blue, 226 }  ,draw opacity=1 ][fill={rgb, 255:red, 74; green, 144; blue, 226 }  ,fill opacity=1 ] (286,104.25) .. controls (286,103.56) and (286.56,103) .. (287.25,103) .. controls (287.94,103) and (288.5,103.56) .. (288.5,104.25) .. controls (288.5,104.94) and (287.94,105.5) .. (287.25,105.5) .. controls (286.56,105.5) and (286,104.94) .. (286,104.25) -- cycle ;
%Shape: Circle [id:dp7249755967031284] 
\draw  [color={rgb, 255:red, 74; green, 144; blue, 226 }  ,draw opacity=1 ][fill={rgb, 255:red, 74; green, 144; blue, 226 }  ,fill opacity=1 ] (305.25,111) .. controls (305.25,110.31) and (305.81,109.75) .. (306.5,109.75) .. controls (307.19,109.75) and (307.75,110.31) .. (307.75,111) .. controls (307.75,111.69) and (307.19,112.25) .. (306.5,112.25) .. controls (305.81,112.25) and (305.25,111.69) .. (305.25,111) -- cycle ;
%Curve Lines [id:da974992303241054] 
\draw    (336.33,97.17) .. controls (341.5,102.17) and (309.17,105.33) .. (306.5,111) ;
%Curve Lines [id:da7238860323634725] 
\draw    (336.33,97.17) .. controls (341.17,99.33) and (345.67,105.67) .. (340.5,110.25) ;
%Curve Lines [id:da12282902785168526] 
\draw    (321.33,104.67) .. controls (323.83,107.17) and (323.33,111.67) .. (320.83,113.17) ;
%Curve Lines [id:da0834593175664392] 
\draw  [dash pattern={on 0.84pt off 2.51pt}]  (321.33,104.67) .. controls (319,107) and (318.83,110.67) .. (320.83,113.17) ;
%Curve Lines [id:da7911648588249378] 
\draw    (286.94,82.82) .. controls (278.47,88.47) and (291.42,91.32) .. (287.25,104.25) ;
%Curve Lines [id:da537136858910889] 
\draw    (286.94,82.82) .. controls (282.68,82.16) and (279.42,86.79) .. (280.5,90.25) ;
%Curve Lines [id:da1677657845720114] 
\draw  [dash pattern={on 0.84pt off 2.51pt}]  (285.11,90.68) .. controls (287.09,93.01) and (285.08,96.42) .. (282.48,97.53) ;
%Curve Lines [id:da9792472310676648] 
\draw    (285.11,90.68) .. controls (282.99,91.63) and (280.59,95.18) .. (282.48,97.53) ;
%Curve Lines [id:da05565010484550614] 
\draw    (331.5,90.77) .. controls (326.86,93.32) and (327.86,100.68) .. (328.59,102.59) ;
%Curve Lines [id:da7985384115144029] 
\draw  [dash pattern={on 0.84pt off 2.51pt}]  (330.83,106.96) .. controls (329.68,105.59) and (329.05,105.68) .. (328.59,102.68) ;
%Curve Lines [id:da36356329845207525] 
\draw  [dash pattern={on 0.84pt off 2.51pt}]  (330.83,106.87) .. controls (334.86,102.95) and (335.14,93.41) .. (331.5,90.77) ;
%Curve Lines [id:da09583645938042984] 
\draw    (289.86,80.09) .. controls (287.32,82.95) and (287.5,91.59) .. (289.68,93.32) ;
%Curve Lines [id:da8230847415007865] 
\draw  [dash pattern={on 0.84pt off 2.51pt}]  (289.68,93.32) .. controls (293.71,89.41) and (293.5,82.72) .. (289.86,80.09) ;
%Straight Lines [id:da8136577929485114] 
\draw [color={rgb, 255:red, 74; green, 144; blue, 226 }  ,draw opacity=1 ][line width=0.75]  [dash pattern={on 0.84pt off 2.51pt}]  (511.17,103.8) -- (530.5,110.38) ;
%Straight Lines [id:da6678558850742716] 
\draw    (380,110.13) -- (428,110.13) ;
\draw [shift={(430,110.13)}, rotate = 180] [color={rgb, 255:red, 0; green, 0; blue, 0 }  ][line width=0.75]    (10.93,-3.29) .. controls (6.95,-1.4) and (3.31,-0.3) .. (0,0) .. controls (3.31,0.3) and (6.95,1.4) .. (10.93,3.29)   ;
%Curve Lines [id:da9278376823463008] 
\draw    (517.6,83.13) .. controls (520.83,87.3) and (522.17,93.13) .. (526.33,97.3) ;
%Curve Lines [id:da9041266903061229] 
\draw    (476.94,82.96) .. controls (481.06,79.9) and (480.59,77.31) .. (484.05,75.53) ;
%Curve Lines [id:da6712101912773961] 
\draw  [dash pattern={on 0.84pt off 2.51pt}]  (499.47,99.8) .. controls (505.35,96.13) and (507.55,88.13) .. (508.35,84.33) ;
%Curve Lines [id:da023432283953845445] 
\draw    (450.14,99.8) .. controls (451.69,90.24) and (465.25,88.46) .. (470.5,90.38) ;
%Straight Lines [id:da8840553126513105] 
\draw [color={rgb, 255:red, 74; green, 144; blue, 226 }  ,draw opacity=1 ][line width=0.75]    (470.5,90.38) -- (511.17,103.8) ;
%Curve Lines [id:da6856561806738547] 
\draw    (509.92,131.13) .. controls (514.14,114.91) and (523.47,120.24) .. (530.5,110.38) ;
%Curve Lines [id:da31266432958895174] 
\draw    (449.92,111.13) .. controls (454.14,94.91) and (463.47,100.24) .. (470.5,90.38) ;
%Straight Lines [id:da10411387967987684] 
\draw  [dash pattern={on 4.5pt off 4.5pt}]  (449.92,111.13) -- (509.92,131.13) ;
%Curve Lines [id:da8069189768174408] 
\draw    (535.25,113.13) .. controls (539.92,115.13) and (547.47,111.58) .. (550.36,120.24) ;
%Straight Lines [id:da5166251735740195] 
\draw  [dash pattern={on 0.84pt off 2.51pt}]  (528.14,113.58) -- (550.36,120.24) ;
%Curve Lines [id:da6136283989016535] 
\draw  [dash pattern={on 4.5pt off 4.5pt}]  (509.92,131.13) .. controls (522.36,138.46) and (557.47,128.02) .. (550.36,120.24) ;
%Curve Lines [id:da537013687972558] 
\draw    (541.03,130.02) .. controls (540.14,119.13) and (537.03,111.8) .. (530.5,110.38) ;
%Curve Lines [id:da03567967389160376] 
\draw  [dash pattern={on 4.5pt off 4.5pt}]  (449.92,111.13) .. controls (448.14,106.91) and (448.14,102.24) .. (450.14,99.8) ;
%Curve Lines [id:da5155817164719191] 
\draw    (499.47,99.8) .. controls (490.81,94.24) and (493.15,85.53) .. (493.35,81.13) ;
%Curve Lines [id:da03361218527411658] 
\draw  [dash pattern={on 4.5pt off 4.5pt}]  (484.05,75.53) .. controls (485.55,81.33) and (511.35,86.13) .. (517.6,83.13) ;
%Curve Lines [id:da5107046357000468] 
\draw  [dash pattern={on 4.5pt off 4.5pt}]  (484.05,75.53) .. controls (491.75,70.73) and (522.95,77.73) .. (517.6,83.13) ;
%Curve Lines [id:da020991979273765726] 
\draw    (508.35,84.33) .. controls (508.95,82.93) and (509.35,80.93) .. (509.15,76.93) ;
%Curve Lines [id:da022508363150097188] 
\draw [color={rgb, 255:red, 74; green, 144; blue, 226 }  ,draw opacity=1 ]   (470.5,90.38) .. controls (469.37,108.58) and (511.89,118.47) .. (530.5,110.38) ;
%Curve Lines [id:da6991600512751911] 
\draw    (526.33,97.3) .. controls (531.17,99.46) and (535.67,105.8) .. (530.5,110.38) ;
%Curve Lines [id:da8484688652588317] 
\draw    (496.28,102.48) .. controls (493.47,104.79) and (493.16,108.58) .. (495.78,110.98) ;
%Curve Lines [id:da6419095479422362] 
\draw    (476.94,82.96) .. controls (472.68,82.29) and (469.42,86.92) .. (470.5,90.38) ;
%Curve Lines [id:da7424746215266707] 
\draw  [dash pattern={on 0.84pt off 2.51pt}]  (480.95,96.63) .. controls (483.47,99.05) and (481.37,103.74) .. (480.11,104.47) ;
%Curve Lines [id:da1420889391332868] 
\draw    (480.95,96.63) .. controls (479.26,97.47) and (477.68,102.11) .. (480.11,104.47) ;
%Curve Lines [id:da9387585697605212] 
\draw    (521.5,90.9) .. controls (516.86,93.45) and (517.86,100.81) .. (518.59,102.72) ;
%Curve Lines [id:da6938188204861125] 
\draw  [dash pattern={on 0.84pt off 2.51pt}]  (520.83,107.09) .. controls (519.68,105.72) and (519.05,105.81) .. (518.59,102.81) ;
%Curve Lines [id:da4490775663080052] 
\draw  [dash pattern={on 0.84pt off 2.51pt}]  (520.83,107) .. controls (524.86,103.09) and (525.14,93.54) .. (521.5,90.9) ;
%Curve Lines [id:da15525336641095044] 
\draw    (479.86,80.22) .. controls (477.32,83.09) and (477.5,91.72) .. (479.68,93.45) ;
%Curve Lines [id:da46724405776064226] 
\draw  [dash pattern={on 0.84pt off 2.51pt}]  (479.68,93.45) .. controls (483.71,89.54) and (483.5,82.86) .. (479.86,80.22) ;
%Curve Lines [id:da0993917771969165] 
\draw    (526.33,97.3) .. controls (526.27,114.33) and (456,95.53) .. (476.94,82.96) ;

% Text Node
\draw (133,95) node [anchor=north west][inner sep=0.75pt]  [font=\footnotesize,color={rgb, 255:red, 74; green, 144; blue, 226 }  ,opacity=1 ]  {$\alpha $};
% Text Node
\draw (78.33,98.4) node [anchor=north west][inner sep=0.75pt]  [font=\footnotesize,color={rgb, 255:red, 208; green, 2; blue, 27 }  ,opacity=1 ]  {$\alpha '$};

\end{tikzpicture}
    \caption{A Roseman move of type d}
    \label{fig:Rosed}
\end{figure}
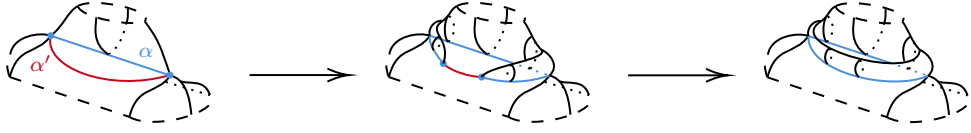

\section{Crossing Numbers for Surface Knots}

We seek to understand the triple points of a 2-knot $K$'s immersion into $S^3$. We will first require two lemmas controlling the number of RIII-moves used to untie low-crossing unlink diagrams.

\begin{lemma}

Every unlink diagram with at most four crossings can be untied using no RIII-moves.

\label{lemma:four_crossing_one_RIII}
\end{lemma}

\begin{proof}[Proof]

We will partition $n$-crossing unlink diagrams by their corresponding 4-valent (multi-) graph, i.e., by forgetting crossing information.

Note that loops in a diagram's graph correspond to half-twists eliminable by a single RI-move. Further, disconnected graphs may be viewed as the disjoint union of graphs for diagrams with fewer crossings, and so we will restrict to connected, loopless graphs.

The case of $n \le 2$ is clear, as an RIII-move requires three crossings. We will then begin with 3-crossing diagrams.

Referring to \cite{oeis}, we see that there is exactly one connected, loopless 4-valent graph on three vertices (see Figure \ref{fig:three_vertex_four_regular_graphs}). It is also easy to see that the figure depicts the only embedding of the graph into $S^2$, since the two edges between vertices 1 and 2 form a bigon and vertex 3 can be on either side of it (in $S^2$, the two sides are symmetric). Once a side is chosen for vertex 3, there are no choices for where to put the two edges connecting vertices 2 and 3. Then, it is easy to see that any way the edges connecting 1 and 3 are placed, one obtains a graph equivalent in $S^2$ to the one depicted.

If crossings are added such that the diagram is alternating, we obtain a diagram for a trefoil. Otherwise, two adjacent crossings will share an overstrand, and therefore also an understrand. An RII-move will reduce the diagram to one with a single crossing.

\begin{figure}[h!]
    \centering

\tikzset{every picture/.style={line width=0.75pt}} %set default line width to 0.75pt        

\begin{tikzpicture}[x=0.75pt,y=0.75pt,yscale=-1,xscale=1]
%uncomment if require: \path (0,191); %set diagram left start at 0, and has height of 191

%Flowchart: Connector [id:dp4322414751918904] 
\draw  [fill={rgb, 255:red, 0; green, 0; blue, 0 }  ,fill opacity=1 ] (68,41.33) .. controls (68,40.23) and (68.9,39.33) .. (70,39.33) .. controls (71.1,39.33) and (72,40.23) .. (72,41.33) .. controls (72,42.44) and (71.1,43.33) .. (70,43.33) .. controls (68.9,43.33) and (68,42.44) .. (68,41.33) -- cycle ;
%Curve Lines [id:da9904535914208306] 
\draw    (30,111.33) .. controls (30.17,81.44) and (40.21,51.48) .. (70,41.33) ;
%Flowchart: Connector [id:dp6613491350056229] 
\draw  [fill={rgb, 255:red, 0; green, 0; blue, 0 }  ,fill opacity=1 ] (28,111.33) .. controls (28,110.23) and (28.9,109.33) .. (30,109.33) .. controls (31.1,109.33) and (32,110.23) .. (32,111.33) .. controls (32,112.44) and (31.1,113.33) .. (30,113.33) .. controls (28.9,113.33) and (28,112.44) .. (28,111.33) -- cycle ;
%Curve Lines [id:da08953798060312357] 
\draw    (110,111.33) .. controls (110.14,81.52) and (100.21,51.48) .. (70,41.33) ;
%Flowchart: Connector [id:dp7191458689307756] 
\draw  [fill={rgb, 255:red, 0; green, 0; blue, 0 }  ,fill opacity=1 ] (108,111.33) .. controls (108,110.23) and (108.9,109.33) .. (110,109.33) .. controls (111.1,109.33) and (112,110.23) .. (112,111.33) .. controls (112,112.44) and (111.1,113.33) .. (110,113.33) .. controls (108.9,113.33) and (108,112.44) .. (108,111.33) -- cycle ;
%Curve Lines [id:da9136397821252505] 
\draw    (30,111.33) .. controls (39.79,126.69) and (100.21,126.69) .. (110,111.33) ;
%Straight Lines [id:da43683502444182765] 
\draw    (30,111.33) -- (110,111.33) ;
%Straight Lines [id:da14102834966976163] 
\draw    (30,111.33) -- (70,41.33) ;
%Straight Lines [id:da25369560530423874] 
\draw    (70,41.33) -- (110,111.33) ;

% Text Node
\draw (71,37.93) node [anchor=south] [inner sep=0.75pt]    {$1$};
% Text Node
\draw (30,114.73) node [anchor=north east] [inner sep=0.75pt]    {$2$};
% Text Node
\draw (112,114.73) node [anchor=north west][inner sep=0.75pt]    {$3$};

\end{tikzpicture}
    \caption{A graph representing all connected, loopless three-crossing link diagrams}
    \label{fig:three_vertex_four_regular_graphs}
\end{figure}
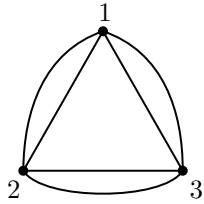

The four-crossing case includes three diagrams (Figure \ref{fig:four_vertex_four_regular_graphs}). We will study each of these and see that if it is the graph underlying a four-crossing unlink projection, then that unlink can be untied without RIII moves. 

First, observe that even though a graph may contain multiple nonequivalent embeddings into $S^2$, these particular graphs do not have other nonequivalent embeddings: In the case of 4A, the three edges between vertices 1 and 4 divide $S^2$ into three regions, which are symmetric to each other. Then vertices 2 and 3 must be in one of these regions. Once this region is chosen, there are no choices for where to place the three edges between vertices 2 and 3, how to connect vertices 1 and 2 to each other and vertices 3 and 4 to each other. The cases of 4B and 4C are similar, except that there are two edges between vertices 1 and 4 that divide the sphere into two regions, one of which contains vertices 2 and 3. A quick check of the cases shows that there are no nonequivalent embeddings.

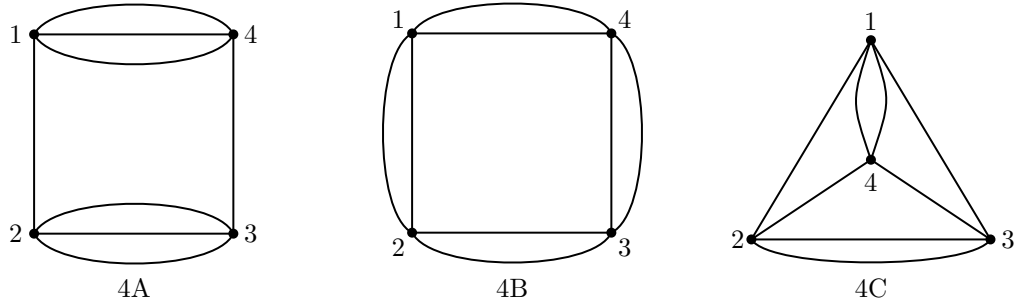
\begin{figure}
    \centering

\tikzset{every picture/.style={line width=0.75pt}} %set default line width to 0.75pt        

\begin{tikzpicture}[x=0.75pt,y=0.75pt,yscale=-1,xscale=1]
%uncomment if require: \path (0,213); %set diagram left start at 0, and has height of 213

%Flowchart: Connector [id:dp8483638101559281] 
\draw  [fill={rgb, 255:red, 0; green, 0; blue, 0 }  ,fill opacity=1 ] (488.27,43.4) .. controls (488.27,42.3) and (489.16,41.4) .. (490.27,41.4) .. controls (491.37,41.4) and (492.27,42.3) .. (492.27,43.4) .. controls (492.27,44.5) and (491.37,45.4) .. (490.27,45.4) .. controls (489.16,45.4) and (488.27,44.5) .. (488.27,43.4) -- cycle ;
%Flowchart: Connector [id:dp17574295757869418] 
\draw  [fill={rgb, 255:red, 0; green, 0; blue, 0 }  ,fill opacity=1 ] (428.27,143.4) .. controls (428.27,142.3) and (429.16,141.4) .. (430.27,141.4) .. controls (431.37,141.4) and (432.27,142.3) .. (432.27,143.4) .. controls (432.27,144.5) and (431.37,145.4) .. (430.27,145.4) .. controls (429.16,145.4) and (428.27,144.5) .. (428.27,143.4) -- cycle ;
%Flowchart: Connector [id:dp6889317826291567] 
\draw  [fill={rgb, 255:red, 0; green, 0; blue, 0 }  ,fill opacity=1 ] (548.27,143.4) .. controls (548.27,142.3) and (549.16,141.4) .. (550.27,141.4) .. controls (551.37,141.4) and (552.27,142.3) .. (552.27,143.4) .. controls (552.27,144.5) and (551.37,145.4) .. (550.27,145.4) .. controls (549.16,145.4) and (548.27,144.5) .. (548.27,143.4) -- cycle ;
%Curve Lines [id:da3810499779534984] 
\draw    (430.27,143.4) .. controls (440.05,158.76) and (540.48,158.76) .. (550.27,143.4) ;
%Straight Lines [id:da43653651981932085] 
\draw    (430.27,143.4) -- (550.27,143.4) ;
%Straight Lines [id:da8321852513012561] 
\draw    (430.27,143.4) -- (490.27,43.4) ;
%Straight Lines [id:da7980552271615796] 
\draw    (490.27,43.4) -- (550.27,143.4) ;
%Flowchart: Connector [id:dp6776888608574723] 
\draw  [fill={rgb, 255:red, 0; green, 0; blue, 0 }  ,fill opacity=1 ] (488.27,103.4) .. controls (488.27,102.3) and (489.16,101.4) .. (490.27,101.4) .. controls (491.37,101.4) and (492.27,102.3) .. (492.27,103.4) .. controls (492.27,104.5) and (491.37,105.4) .. (490.27,105.4) .. controls (489.16,105.4) and (488.27,104.5) .. (488.27,103.4) -- cycle ;
%Straight Lines [id:da7042374818849377] 
\draw    (430.27,143.4) -- (490.27,103.4) ;
%Straight Lines [id:da974860647506673] 
\draw    (490.27,103.4) -- (550.27,143.4) ;
%Curve Lines [id:da8788737090654756] 
\draw    (490.27,103.4) .. controls (480.27,73.9) and (480.27,74.4) .. (490.27,43.4) ;
%Curve Lines [id:da6428469813654594] 
\draw    (490.27,103.4) .. controls (500.27,73.4) and (500.77,73.9) .. (490.27,43.4) ;
%Shape: Rectangle [id:dp3526633922945749] 
\draw   (260,40) -- (360,40) -- (360,140) -- (260,140) -- cycle ;
%Curve Lines [id:da9927264811576506] 
\draw    (260,140) .. controls (270.5,160) and (350.5,160) .. (360,140) ;
%Curve Lines [id:da40599453337017855] 
\draw    (260,40) .. controls (270.5,20) and (350.5,20) .. (360,40) ;
%Curve Lines [id:da11903045376161636] 
\draw    (260,40) .. controls (240.5,50) and (240.5,130) .. (260,140) ;
%Curve Lines [id:da9727862111778248] 
\draw    (360,40) .. controls (380.5,50) and (380.5,130) .. (360,140) ;
%Flowchart: Connector [id:dp9450111377791046] 
\draw  [fill={rgb, 255:red, 0; green, 0; blue, 0 }  ,fill opacity=1 ] (258,140) .. controls (258,138.9) and (258.9,138) .. (260,138) .. controls (261.1,138) and (262,138.9) .. (262,140) .. controls (262,141.1) and (261.1,142) .. (260,142) .. controls (258.9,142) and (258,141.1) .. (258,140) -- cycle ;
%Flowchart: Connector [id:dp39266954511112473] 
\draw  [fill={rgb, 255:red, 0; green, 0; blue, 0 }  ,fill opacity=1 ] (358,140) .. controls (358,138.9) and (358.9,138) .. (360,138) .. controls (361.1,138) and (362,138.9) .. (362,140) .. controls (362,141.1) and (361.1,142) .. (360,142) .. controls (358.9,142) and (358,141.1) .. (358,140) -- cycle ;
%Flowchart: Connector [id:dp22305762446119726] 
\draw  [fill={rgb, 255:red, 0; green, 0; blue, 0 }  ,fill opacity=1 ] (258,40) .. controls (258,38.9) and (258.9,38) .. (260,38) .. controls (261.1,38) and (262,38.9) .. (262,40) .. controls (262,41.1) and (261.1,42) .. (260,42) .. controls (258.9,42) and (258,41.1) .. (258,40) -- cycle ;
%Flowchart: Connector [id:dp9203982379816675] 
\draw  [fill={rgb, 255:red, 0; green, 0; blue, 0 }  ,fill opacity=1 ] (358,40) .. controls (358,38.9) and (358.9,38) .. (360,38) .. controls (361.1,38) and (362,38.9) .. (362,40) .. controls (362,41.1) and (361.1,42) .. (360,42) .. controls (358.9,42) and (358,41.1) .. (358,40) -- cycle ;
%Shape: Rectangle [id:dp5805066928509983] 
\draw   (70.29,40.5) -- (170.29,40.5) -- (170.29,140.5) -- (70.29,140.5) -- cycle ;
%Curve Lines [id:da6027342377945427] 
\draw    (70.29,140.5) .. controls (80.79,160.5) and (160.79,160.5) .. (170.29,140.5) ;
%Curve Lines [id:da261636292638766] 
\draw    (70.29,40.5) .. controls (80.79,20.5) and (160.79,20.5) .. (170.29,40.5) ;
%Flowchart: Connector [id:dp7943392351871669] 
\draw  [fill={rgb, 255:red, 0; green, 0; blue, 0 }  ,fill opacity=1 ] (68.29,140.5) .. controls (68.29,139.4) and (69.19,138.5) .. (70.29,138.5) .. controls (71.4,138.5) and (72.29,139.4) .. (72.29,140.5) .. controls (72.29,141.6) and (71.4,142.5) .. (70.29,142.5) .. controls (69.19,142.5) and (68.29,141.6) .. (68.29,140.5) -- cycle ;
%Flowchart: Connector [id:dp4605517628571669] 
\draw  [fill={rgb, 255:red, 0; green, 0; blue, 0 }  ,fill opacity=1 ] (168.29,140.5) .. controls (168.29,139.4) and (169.19,138.5) .. (170.29,138.5) .. controls (171.4,138.5) and (172.29,139.4) .. (172.29,140.5) .. controls (172.29,141.6) and (171.4,142.5) .. (170.29,142.5) .. controls (169.19,142.5) and (168.29,141.6) .. (168.29,140.5) -- cycle ;
%Flowchart: Connector [id:dp986243517496883] 
\draw  [fill={rgb, 255:red, 0; green, 0; blue, 0 }  ,fill opacity=1 ] (68.29,40.5) .. controls (68.29,39.4) and (69.19,38.5) .. (70.29,38.5) .. controls (71.4,38.5) and (72.29,39.4) .. (72.29,40.5) .. controls (72.29,41.6) and (71.4,42.5) .. (70.29,42.5) .. controls (69.19,42.5) and (68.29,41.6) .. (68.29,40.5) -- cycle ;
%Flowchart: Connector [id:dp18042524516352898] 
\draw  [fill={rgb, 255:red, 0; green, 0; blue, 0 }  ,fill opacity=1 ] (168.29,40.5) .. controls (168.29,39.4) and (169.19,38.5) .. (170.29,38.5) .. controls (171.4,38.5) and (172.29,39.4) .. (172.29,40.5) .. controls (172.29,41.6) and (171.4,42.5) .. (170.29,42.5) .. controls (169.19,42.5) and (168.29,41.6) .. (168.29,40.5) -- cycle ;
%Curve Lines [id:da5838675554208134] 
\draw    (70.29,40.5) .. controls (80.79,60.5) and (160.79,60.5) .. (170.29,40.5) ;
%Curve Lines [id:da45456640777262136] 
\draw    (70.29,140.5) .. controls (80.79,120.5) and (160.79,120.5) .. (170.29,140.5) ;

% Text Node
\draw (110.67,161.67) node [anchor=north west][inner sep=0.75pt]   [align=left] {4A};
% Text Node
\draw (300.67,161.7) node [anchor=north west][inner sep=0.75pt]   [align=left] {4B};
% Text Node
\draw (480.67,161.7) node [anchor=north west][inner sep=0.75pt]   [align=left] {4C};
% Text Node
\draw (66,40.5) node [anchor=east] [inner sep=0.75pt]    {$1$};
% Text Node
\draw (66,140.5) node [anchor=east] [inner sep=0.75pt]    {$2$};
% Text Node
\draw (174.29,140.5) node [anchor=west] [inner sep=0.75pt]    {$3$};
% Text Node
\draw (174.29,40.5) node [anchor=west] [inner sep=0.75pt]    {$4$};
% Text Node
\draw (258,38.6) node [anchor=south east] [inner sep=0.75pt]    {$1$};
% Text Node
\draw (258,141.4) node [anchor=north east] [inner sep=0.75pt]    {$2$};
% Text Node
\draw (362,141.4) node [anchor=north west][inner sep=0.75pt]    {$3$};
% Text Node
\draw (362,38.6) node [anchor=south west] [inner sep=0.75pt]    {$4$};
% Text Node
\draw (490.27,40) node [anchor=south] [inner sep=0.75pt]    {$1$};
% Text Node
\draw (428.27,143.4) node [anchor=east] [inner sep=0.75pt]    {$2$};
% Text Node
\draw (554.27,143.4) node [anchor=west] [inner sep=0.75pt]    {$3$};
% Text Node
\draw (490.27,108.8) node [anchor=north] [inner sep=0.75pt]    {$4$};

\end{tikzpicture}
    \caption{Graphs representing all connected, loopless four-crossing link diagrams}
    \label{fig:four_vertex_four_regular_graphs}
\end{figure}

Notice first that any instance of a tripled edge (Figure \ref{fig:tripled_edge}), as in diagram 4A, will contain a component $M$ of the link after adding crossings. If the vertices sharing the tripled edge are resolved in the same way (that is, so that the overstrand of one becomes the understrand of the other), then $M$ is a meridian of the other component present in the three-strand loop, and therefore the link is not split. Thus, the two vertices must be resolved to share an overstrand, and an RII-move will eliminate both, reducing to a two-crossing diagram.

\begin{figure}[h!]
    \centering

\tikzset{every picture/.style={line width=0.75pt}} %set default line width to 0.75pt        

\begin{tikzpicture}[x=0.75pt,y=0.75pt,yscale=-1,xscale=1]
%uncomment if require: \path (0,91); %set diagram left start at 0, and has height of 91

%Curve Lines [id:da5913366437108214] 
\draw    (50,50) .. controls (60.5,30) and (140.5,30) .. (150,50) ;
%Flowchart: Connector [id:dp43699279633348265] 
\draw  [fill={rgb, 255:red, 0; green, 0; blue, 0 }  ,fill opacity=1 ] (48,50) .. controls (48,48.9) and (48.9,48) .. (50,48) .. controls (51.1,48) and (52,48.9) .. (52,50) .. controls (52,51.1) and (51.1,52) .. (50,52) .. controls (48.9,52) and (48,51.1) .. (48,50) -- cycle ;
%Flowchart: Connector [id:dp7468976115556447] 
\draw  [fill={rgb, 255:red, 0; green, 0; blue, 0 }  ,fill opacity=1 ] (148,50) .. controls (148,48.9) and (148.9,48) .. (150,48) .. controls (151.1,48) and (152,48.9) .. (152,50) .. controls (152,51.1) and (151.1,52) .. (150,52) .. controls (148.9,52) and (148,51.1) .. (148,50) -- cycle ;
%Curve Lines [id:da019573311279188155] 
\draw    (50,50) .. controls (60.5,70) and (140.5,70) .. (150,50) ;
%Straight Lines [id:da9560846211512525] 
\draw    (50,50) -- (150,50) ;
%Straight Lines [id:da32992096279974836] 
\draw    (150,50) -- (160,60) ;
%Straight Lines [id:da5946073023555385] 
\draw  [dash pattern={on 0.84pt off 2.51pt}]  (160,60) -- (170,70) ;
%Straight Lines [id:da19047838862263067] 
\draw    (50,50) -- (40,60) ;
%Straight Lines [id:da921660302501642] 
\draw  [dash pattern={on 0.84pt off 2.51pt}]  (40,60) -- (30,70) ;

\end{tikzpicture}
    \caption{A tripled edge, corresponding a meridian or split component}
    \label{fig:tripled_edge}
\end{figure}
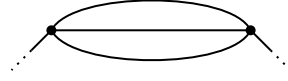

Consider diagram 4B. If crossings are added such that the diagram is alternating, it is a non-trivial link. Thus, an unlink represented by this diagram will have two adjacent crossings sharing either an under- or overstrand (suppose an overstrand). If any two adjacent crossings share an overstrand, perform an RII-move and reduce to a two-crossing diagram. 

For diagram 4C, note that if the diagram is alternating, we obtain the figure eight knot. If crossings 1 and 4 or 2 and 3 share an overstrand, we can reduce via an RII-move. Suppose that neither the upper nor lower pair is eliminable with an RII-move. Then an overstrand is shared between one crossing above and one below. This results in the diagram in Figure \ref{fig:four_crossing_trefoil_diagram} (possibly after a reflection and/or isotopy), which is a diagram of the trefoil. The result now follows.
    
\end{proof}

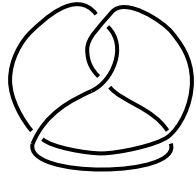
\begin{figure}[h!]
    \centering

\tikzset{every picture/.style={line width=0.75pt}} %set default line width to 0.75pt        

\begin{tikzpicture}[x=0.75pt,y=0.75pt,yscale=-1,xscale=1]
%uncomment if require: \path (0,198); %set diagram left start at 0, and has height of 198

%Curve Lines [id:da5540468481183012] 
\draw [line width=1.5]    (60.5,96.58) .. controls (56.13,114.07) and (135.83,113.42) .. (129.83,96.42) ;
%Curve Lines [id:da5196780984678055] 
\draw [line width=1.5]    (90,70) .. controls (100.44,65.44) and (108.42,47.59) .. (98.2,37.7) ;
%Curve Lines [id:da9852689038721693] 
\draw [line width=1.5]    (99.23,68.07) .. controls (105.23,75.47) and (120.83,78.42) .. (128.17,89.92) ;
%Curve Lines [id:da6375264813241613] 
\draw [line width=1.5]    (93.65,62.73) .. controls (89.8,58.93) and (87.91,55) .. (87.91,49.18) .. controls (87.91,43.36) and (93.5,38.41) .. (100.25,30.43) .. controls (107,22.45) and (126.03,35.21) .. (130.25,40.43) .. controls (134.47,45.65) and (140.45,53.13) .. (140.45,65.13) .. controls (140.45,77.13) and (134.3,88.52) .. (128.5,92.92) .. controls (122.7,97.32) and (103,101.31) .. (95.25,101.33) .. controls (87.5,101.36) and (69.67,98.24) .. (65.45,94.13) ;
%Curve Lines [id:da22170562731833043] 
\draw [line width=1.5]    (60,90) .. controls (55.78,85) and (49.2,74.5) .. (49.2,64.5) .. controls (49.2,54.5) and (54.67,45) .. (60,40) .. controls (65.33,35) and (82.2,18.1) .. (92.4,31.5) ;
%Curve Lines [id:da7694444091129996] 
\draw [line width=1.5]    (60.5,96.58) .. controls (66.83,80.92) and (80.22,74.56) .. (90,70) ;

\end{tikzpicture}
    \caption{A diagram of the trefoil}
    \label{fig:four_crossing_trefoil_diagram}
\end{figure}

\begin{lemma}

Any unlink diagram with five crossings can be untied with at most one RIII-move.
\label{lemma:fivecross}
\end{lemma}

\begin{proof}[Proof]

From \cite{oeis}, we know there are exactly six loopless 4-valent graphs on five vertices. One of these is $K_5$, the complete graph on five vertices, which is not planar, and therefore has no knot diagram representatives. The other five graphs appear in Figure \ref{fig:five_crossing_four_regular_graphs}.

Similarly to in the proof of the previous lemma, let us first check that for graphs 5C, 5D, and 5E, there are no nonequivalent planar embeddings for these graphs. (For 5A and 5B, we do not actually need to check this, because it is easy to see that any planar embedding needs to have the tripled edge, which, if the graph underlies an unlink, can always be eliminated by an RII move that pulls out a component.)

For the other graphs, note that each doubled edge forms a bigon that splits $S^2$ into two regions, and all the other vertices must be in the same region. Thus, if we fill in the unoccupied region, we may think of the bigons as operating as a single thick edge. From here a case-by-case shows that there are no other embeddings: for 5C, the graph consists of a single cycle of 5 vertices and 5 thick edges. For 5D and 5E, note that vertices 1, 2, 3, 4 form a cycle (of regular and thick edges) that divides the $S^2$ into two symmetric regions. In one of them, we place vertex 5,  and then there is only one way to connect it to the other vertices.

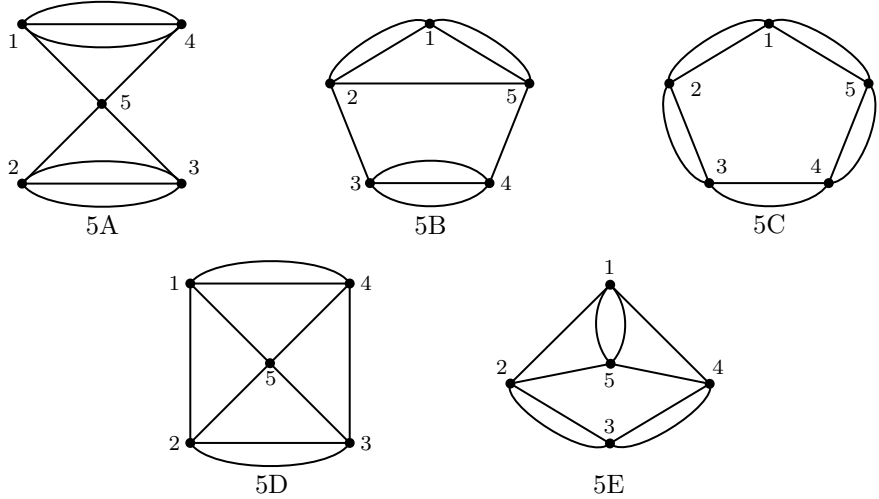
\begin{figure}[h!]
    \centering

\tikzset{every picture/.style={line width=0.75pt}} %set default line width to 0.75pt        

\begin{tikzpicture}[x=0.75pt,y=0.75pt,yscale=-1,xscale=1]
%uncomment if require: \path (0,306); %set diagram left start at 0, and has height of 306

%Shape: Circle [id:dp9640238986732845] 
\draw  [fill={rgb, 255:red, 0; green, 0; blue, 0 }  ,fill opacity=1 ] (408,110) .. controls (408,108.9) and (408.9,108) .. (410,108) .. controls (411.1,108) and (412,108.9) .. (412,110) .. controls (412,111.1) and (411.1,112) .. (410,112) .. controls (408.9,112) and (408,111.1) .. (408,110) -- cycle ;
%Straight Lines [id:da8787713366791375] 
\draw    (410,110) -- (470,110) ;
%Shape: Circle [id:dp9445874225407801] 
\draw  [fill={rgb, 255:red, 0; green, 0; blue, 0 }  ,fill opacity=1 ] (468,110) .. controls (468,108.9) and (468.9,108) .. (470,108) .. controls (471.1,108) and (472,108.9) .. (472,110) .. controls (472,111.1) and (471.1,112) .. (470,112) .. controls (468.9,112) and (468,111.1) .. (468,110) -- cycle ;
%Straight Lines [id:da11678964417533266] 
\draw    (390,60) -- (440,30) ;
%Straight Lines [id:da8689321810127201] 
\draw    (490,60) -- (440,30) ;
%Straight Lines [id:da4072609390301055] 
\draw    (410,110) -- (390,60) ;
%Straight Lines [id:da6328465718585151] 
\draw    (470,110) -- (490,60) ;
%Shape: Circle [id:dp8065254497434383] 
\draw  [fill={rgb, 255:red, 0; green, 0; blue, 0 }  ,fill opacity=1 ] (488,60) .. controls (488,58.9) and (488.9,58) .. (490,58) .. controls (491.1,58) and (492,58.9) .. (492,60) .. controls (492,61.1) and (491.1,62) .. (490,62) .. controls (488.9,62) and (488,61.1) .. (488,60) -- cycle ;
%Shape: Circle [id:dp7814214329344622] 
\draw  [fill={rgb, 255:red, 0; green, 0; blue, 0 }  ,fill opacity=1 ] (388,60) .. controls (388,58.9) and (388.9,58) .. (390,58) .. controls (391.1,58) and (392,58.9) .. (392,60) .. controls (392,61.1) and (391.1,62) .. (390,62) .. controls (388.9,62) and (388,61.1) .. (388,60) -- cycle ;
%Shape: Circle [id:dp22494806359769715] 
\draw  [fill={rgb, 255:red, 0; green, 0; blue, 0 }  ,fill opacity=1 ] (438,30) .. controls (438,28.9) and (438.9,28) .. (440,28) .. controls (441.1,28) and (442,28.9) .. (442,30) .. controls (442,31.1) and (441.1,32) .. (440,32) .. controls (438.9,32) and (438,31.1) .. (438,30) -- cycle ;
%Curve Lines [id:da9082396957362973] 
\draw    (410,110) .. controls (420.33,125.5) and (460.33,125.5) .. (470,110) ;
%Curve Lines [id:da7699575922227414] 
\draw    (470,110) .. controls (485.67,110.17) and (500.33,70.17) .. (490,60) ;
%Curve Lines [id:da011795636642213503] 
\draw    (410,110) .. controls (395,109.83) and (380.33,70.17) .. (390,60) ;
%Curve Lines [id:da2543037777284507] 
\draw    (390,60) .. controls (385.33,49.83) and (430,20.5) .. (440,30) ;
%Curve Lines [id:da07474124447725572] 
\draw    (490,60) .. controls (495.33,49.83) and (450.33,20.17) .. (440,30) ;
%Shape: Circle [id:dp11502103979164213] 
\draw  [fill={rgb, 255:red, 0; green, 0; blue, 0 }  ,fill opacity=1 ] (237.77,110) .. controls (237.77,108.9) and (238.67,108) .. (239.77,108) .. controls (240.87,108) and (241.77,108.9) .. (241.77,110) .. controls (241.77,111.1) and (240.87,112) .. (239.77,112) .. controls (238.67,112) and (237.77,111.1) .. (237.77,110) -- cycle ;
%Straight Lines [id:da7967998054950515] 
\draw    (239.77,110) -- (299.77,110) ;
%Shape: Circle [id:dp3313325058466662] 
\draw  [fill={rgb, 255:red, 0; green, 0; blue, 0 }  ,fill opacity=1 ] (297.77,110) .. controls (297.77,108.9) and (298.67,108) .. (299.77,108) .. controls (300.87,108) and (301.77,108.9) .. (301.77,110) .. controls (301.77,111.1) and (300.87,112) .. (299.77,112) .. controls (298.67,112) and (297.77,111.1) .. (297.77,110) -- cycle ;
%Straight Lines [id:da5552997668297112] 
\draw    (219.77,60) -- (269.77,30) ;
%Straight Lines [id:da892030864325264] 
\draw    (319.77,60) -- (269.77,30) ;
%Straight Lines [id:da015311647961891217] 
\draw    (239.77,110) -- (219.77,60) ;
%Straight Lines [id:da3377999895258773] 
\draw    (299.77,110) -- (319.77,60) ;
%Shape: Circle [id:dp42143824624181336] 
\draw  [fill={rgb, 255:red, 0; green, 0; blue, 0 }  ,fill opacity=1 ] (317.77,60) .. controls (317.77,58.9) and (318.67,58) .. (319.77,58) .. controls (320.87,58) and (321.77,58.9) .. (321.77,60) .. controls (321.77,61.1) and (320.87,62) .. (319.77,62) .. controls (318.67,62) and (317.77,61.1) .. (317.77,60) -- cycle ;
%Shape: Circle [id:dp48025825965172875] 
\draw  [fill={rgb, 255:red, 0; green, 0; blue, 0 }  ,fill opacity=1 ] (217.77,60) .. controls (217.77,58.9) and (218.67,58) .. (219.77,58) .. controls (220.87,58) and (221.77,58.9) .. (221.77,60) .. controls (221.77,61.1) and (220.87,62) .. (219.77,62) .. controls (218.67,62) and (217.77,61.1) .. (217.77,60) -- cycle ;
%Shape: Circle [id:dp7682183784066424] 
\draw  [fill={rgb, 255:red, 0; green, 0; blue, 0 }  ,fill opacity=1 ] (267.77,30) .. controls (267.77,28.9) and (268.67,28) .. (269.77,28) .. controls (270.87,28) and (271.77,28.9) .. (271.77,30) .. controls (271.77,31.1) and (270.87,32) .. (269.77,32) .. controls (268.67,32) and (267.77,31.1) .. (267.77,30) -- cycle ;
%Curve Lines [id:da5791932540553008] 
\draw    (239.77,110) .. controls (250.1,125.5) and (290.1,125.5) .. (299.77,110) ;
%Curve Lines [id:da5573305794979888] 
\draw    (219.77,60) .. controls (215.1,49.83) and (259.77,20.5) .. (269.77,30) ;
%Curve Lines [id:da009419725989958327] 
\draw    (319.77,60) .. controls (325.1,49.83) and (280.1,20.17) .. (269.77,30) ;
%Curve Lines [id:da10516020581296037] 
\draw    (239.77,110) .. controls (250,95.17) and (290.33,94.83) .. (299.77,110) ;
%Straight Lines [id:da27713360057532077] 
\draw    (219.77,60) -- (319.77,60) ;
%Straight Lines [id:da6300627648663941] 
\draw    (149.6,240.4) -- (149.6,160.4) ;
%Straight Lines [id:da4879120943339804] 
\draw    (229.6,240.4) -- (229.6,160.4) ;
%Straight Lines [id:da7279857052025339] 
\draw    (149.6,240.4) -- (229.6,240.4) ;
%Straight Lines [id:da4740710426514254] 
\draw    (149.6,160.4) -- (229.6,160.4) ;
%Straight Lines [id:da9858437500383119] 
\draw    (149.6,240.4) -- (189.6,200.4) ;
%Straight Lines [id:da015936888816523553] 
\draw    (229.6,240.4) -- (189.6,200.4) ;
%Straight Lines [id:da859985726504046] 
\draw    (189.6,200.4) -- (229.6,160.4) ;
%Straight Lines [id:da7413484667097071] 
\draw    (189.6,200.4) -- (149.6,160.4) ;
%Curve Lines [id:da7875355091703427] 
\draw    (149.6,240.4) .. controls (159.93,255.9) and (219.93,255.9) .. (229.6,240.4) ;
%Curve Lines [id:da20540921274346002] 
\draw    (149.6,160.4) .. controls (159.83,145.57) and (220.16,145.23) .. (229.6,160.4) ;
%Shape: Circle [id:dp9489326019566212] 
\draw  [fill={rgb, 255:red, 0; green, 0; blue, 0 }  ,fill opacity=1 ] (147.6,160.4) .. controls (147.6,159.3) and (148.5,158.4) .. (149.6,158.4) .. controls (150.7,158.4) and (151.6,159.3) .. (151.6,160.4) .. controls (151.6,161.5) and (150.7,162.4) .. (149.6,162.4) .. controls (148.5,162.4) and (147.6,161.5) .. (147.6,160.4) -- cycle ;
%Shape: Circle [id:dp42868387792361307] 
\draw  [fill={rgb, 255:red, 0; green, 0; blue, 0 }  ,fill opacity=1 ] (187.6,200.4) .. controls (187.6,199.3) and (188.5,198.4) .. (189.6,198.4) .. controls (190.7,198.4) and (191.6,199.3) .. (191.6,200.4) .. controls (191.6,201.5) and (190.7,202.4) .. (189.6,202.4) .. controls (188.5,202.4) and (187.6,201.5) .. (187.6,200.4) -- cycle ;
%Shape: Circle [id:dp26710052973554965] 
\draw  [fill={rgb, 255:red, 0; green, 0; blue, 0 }  ,fill opacity=1 ] (227.6,160.4) .. controls (227.6,159.3) and (228.5,158.4) .. (229.6,158.4) .. controls (230.7,158.4) and (231.6,159.3) .. (231.6,160.4) .. controls (231.6,161.5) and (230.7,162.4) .. (229.6,162.4) .. controls (228.5,162.4) and (227.6,161.5) .. (227.6,160.4) -- cycle ;
%Shape: Circle [id:dp6015446782359338] 
\draw  [fill={rgb, 255:red, 0; green, 0; blue, 0 }  ,fill opacity=1 ] (147.6,240.4) .. controls (147.6,239.3) and (148.5,238.4) .. (149.6,238.4) .. controls (150.7,238.4) and (151.6,239.3) .. (151.6,240.4) .. controls (151.6,241.5) and (150.7,242.4) .. (149.6,242.4) .. controls (148.5,242.4) and (147.6,241.5) .. (147.6,240.4) -- cycle ;
%Shape: Circle [id:dp5904223998753975] 
\draw  [fill={rgb, 255:red, 0; green, 0; blue, 0 }  ,fill opacity=1 ] (227.6,240.4) .. controls (227.6,239.3) and (228.5,238.4) .. (229.6,238.4) .. controls (230.7,238.4) and (231.6,239.3) .. (231.6,240.4) .. controls (231.6,241.5) and (230.7,242.4) .. (229.6,242.4) .. controls (228.5,242.4) and (227.6,241.5) .. (227.6,240.4) -- cycle ;
%Straight Lines [id:da9027022169235569] 
\draw    (65.2,110.25) -- (145.2,110.25) ;
%Straight Lines [id:da5958148340393429] 
\draw    (65.2,30.25) -- (145.2,30.25) ;
%Straight Lines [id:da3932354577028141] 
\draw    (65.2,110.25) -- (105.2,70.25) ;
%Straight Lines [id:da8796814815704448] 
\draw    (145.2,110.25) -- (105.2,70.25) ;
%Straight Lines [id:da9040612442933468] 
\draw    (105.2,70.25) -- (145.2,30.25) ;
%Straight Lines [id:da9441459141413022] 
\draw    (105.2,70.25) -- (65.2,30.25) ;
%Curve Lines [id:da44691448131188405] 
\draw    (65.2,110.25) .. controls (75.53,125.75) and (135.53,125.75) .. (145.2,110.25) ;
%Curve Lines [id:da5353835871607142] 
\draw    (65.2,30.25) .. controls (75.43,15.42) and (135.76,15.08) .. (145.2,30.25) ;
%Shape: Circle [id:dp700073871933326] 
\draw  [fill={rgb, 255:red, 0; green, 0; blue, 0 }  ,fill opacity=1 ] (63.2,30.25) .. controls (63.2,29.15) and (64.1,28.25) .. (65.2,28.25) .. controls (66.3,28.25) and (67.2,29.15) .. (67.2,30.25) .. controls (67.2,31.35) and (66.3,32.25) .. (65.2,32.25) .. controls (64.1,32.25) and (63.2,31.35) .. (63.2,30.25) -- cycle ;
%Shape: Circle [id:dp5695629920666142] 
\draw  [fill={rgb, 255:red, 0; green, 0; blue, 0 }  ,fill opacity=1 ] (103.2,70.25) .. controls (103.2,69.15) and (104.1,68.25) .. (105.2,68.25) .. controls (106.3,68.25) and (107.2,69.15) .. (107.2,70.25) .. controls (107.2,71.35) and (106.3,72.25) .. (105.2,72.25) .. controls (104.1,72.25) and (103.2,71.35) .. (103.2,70.25) -- cycle ;
%Shape: Circle [id:dp9369162084463412] 
\draw  [fill={rgb, 255:red, 0; green, 0; blue, 0 }  ,fill opacity=1 ] (143.2,30.25) .. controls (143.2,29.15) and (144.1,28.25) .. (145.2,28.25) .. controls (146.3,28.25) and (147.2,29.15) .. (147.2,30.25) .. controls (147.2,31.35) and (146.3,32.25) .. (145.2,32.25) .. controls (144.1,32.25) and (143.2,31.35) .. (143.2,30.25) -- cycle ;
%Shape: Circle [id:dp0928876892723961] 
\draw  [fill={rgb, 255:red, 0; green, 0; blue, 0 }  ,fill opacity=1 ] (63.2,110.25) .. controls (63.2,109.15) and (64.1,108.25) .. (65.2,108.25) .. controls (66.3,108.25) and (67.2,109.15) .. (67.2,110.25) .. controls (67.2,111.35) and (66.3,112.25) .. (65.2,112.25) .. controls (64.1,112.25) and (63.2,111.35) .. (63.2,110.25) -- cycle ;
%Shape: Circle [id:dp9086815195658415] 
\draw  [fill={rgb, 255:red, 0; green, 0; blue, 0 }  ,fill opacity=1 ] (143.2,110.25) .. controls (143.2,109.15) and (144.1,108.25) .. (145.2,108.25) .. controls (146.3,108.25) and (147.2,109.15) .. (147.2,110.25) .. controls (147.2,111.35) and (146.3,112.25) .. (145.2,112.25) .. controls (144.1,112.25) and (143.2,111.35) .. (143.2,110.25) -- cycle ;
%Curve Lines [id:da7280349599660244] 
\draw    (65.2,110.25) .. controls (75.43,95.42) and (135.76,95.08) .. (145.2,110.25) ;
%Curve Lines [id:da3997018376408177] 
\draw    (65.2,30.25) .. controls (75.53,45.75) and (135.53,45.75) .. (145.2,30.25) ;
%Shape: Circle [id:dp23215106227550286] 
\draw  [fill={rgb, 255:red, 0; green, 0; blue, 0 }  ,fill opacity=1 ] (362.54,200.67) .. controls (362.54,201.77) and (361.64,202.67) .. (360.54,202.67) .. controls (359.43,202.67) and (358.54,201.77) .. (358.54,200.66) .. controls (358.54,199.56) and (359.44,198.67) .. (360.54,198.67) .. controls (361.65,198.67) and (362.54,199.57) .. (362.54,200.67) -- cycle ;
%Shape: Circle [id:dp33190005575149584] 
\draw  [fill={rgb, 255:red, 0; green, 0; blue, 0 }  ,fill opacity=1 ] (362.33,161) .. controls (362.32,162.11) and (361.43,163) .. (360.32,163) .. controls (359.22,163) and (358.32,162.1) .. (358.33,161) .. controls (358.33,159.89) and (359.22,159) .. (360.33,159) .. controls (361.43,159) and (362.33,159.9) .. (362.33,161) -- cycle ;
%Straight Lines [id:da5672558379137075] 
\draw    (410.26,210.74) -- (360.21,240.67) ;
%Straight Lines [id:da6976529915587796] 
\draw    (310.26,210.6) -- (360.21,240.67) ;
%Straight Lines [id:da3606669306654191] 
\draw    (360.54,200.67) -- (410.26,210.74) ;
%Straight Lines [id:da4634674305133495] 
\draw    (360.54,200.67) -- (310.26,210.6) ;
%Shape: Circle [id:dp9186463066672925] 
\draw  [fill={rgb, 255:red, 0; green, 0; blue, 0 }  ,fill opacity=1 ] (312.26,210.6) .. controls (312.25,211.7) and (311.36,212.6) .. (310.25,212.6) .. controls (309.15,212.6) and (308.25,211.7) .. (308.26,210.59) .. controls (308.26,209.49) and (309.15,208.6) .. (310.26,208.6) .. controls (311.36,208.6) and (312.26,209.5) .. (312.26,210.6) -- cycle ;
%Shape: Circle [id:dp8341691829372796] 
\draw  [fill={rgb, 255:red, 0; green, 0; blue, 0 }  ,fill opacity=1 ] (412.26,210.74) .. controls (412.25,211.84) and (411.36,212.74) .. (410.25,212.74) .. controls (409.15,212.73) and (408.25,211.84) .. (408.26,210.73) .. controls (408.26,209.63) and (409.15,208.73) .. (410.26,208.74) .. controls (411.36,208.74) and (412.26,209.63) .. (412.26,210.74) -- cycle ;
%Shape: Circle [id:dp9166159425061825] 
\draw  [fill={rgb, 255:red, 0; green, 0; blue, 0 }  ,fill opacity=1 ] (362.21,240.67) .. controls (362.21,241.77) and (361.32,242.67) .. (360.21,242.67) .. controls (359.11,242.67) and (358.21,241.77) .. (358.21,240.66) .. controls (358.22,239.56) and (359.11,238.67) .. (360.22,238.67) .. controls (361.32,238.67) and (362.22,239.56) .. (362.21,240.67) -- cycle ;
%Curve Lines [id:da43436912451855536] 
\draw    (360.54,200.67) .. controls (350.95,190.4) and (350.71,170.9) .. (360.33,161) ;
%Curve Lines [id:da7092951759725261] 
\draw    (410.26,210.74) .. controls (414.91,220.91) and (370.2,250.18) .. (360.21,240.67) ;
%Curve Lines [id:da8758673940869512] 
\draw    (310.26,210.6) .. controls (304.91,220.76) and (349.87,250.49) .. (360.21,240.67) ;
%Curve Lines [id:da9791551451294654] 
\draw    (360.54,200.67) .. controls (370.45,190.68) and (370.21,170.93) .. (360.33,161) ;
%Straight Lines [id:da6410912922943459] 
\draw    (360.33,161) -- (410.26,210.74) ;
%Straight Lines [id:da9358599177370918] 
\draw    (360.33,161) -- (310.26,210.6) ;

% Text Node
\draw (440,33.5) node [anchor=north] [inner sep=0.75pt]  [font=\footnotesize]  {$1$};
% Text Node
\draw (403.5,57.9) node [anchor=north] [inner sep=0.75pt]  [font=\footnotesize]  {$2$};
% Text Node
\draw (412,106.6) node [anchor=south west] [inner sep=0.75pt]  [font=\footnotesize]  {$3$};
% Text Node
\draw (468,106.6) node [anchor=south east] [inner sep=0.75pt]  [font=\footnotesize]  {$4$};
% Text Node
\draw (481,57.9) node [anchor=north] [inner sep=0.75pt]  [font=\footnotesize]  {$5$};
% Text Node
\draw (270,32.5) node [anchor=north] [inner sep=0.75pt]  [font=\footnotesize]  {$1$};
% Text Node
\draw (231,61.9) node [anchor=north] [inner sep=0.75pt]  [font=\footnotesize]  {$2$};
% Text Node
\draw (237,110) node [anchor=east] [inner sep=0.75pt]  [font=\footnotesize]  {$3$};
% Text Node
\draw (303.77,110) node [anchor=west] [inner sep=0.75pt]  [font=\footnotesize]  {$4$};
% Text Node
\draw (311,61.4) node [anchor=north] [inner sep=0.75pt]  [font=\footnotesize]  {$5$};
% Text Node
\draw (146,160.4) node [anchor=east] [inner sep=0.75pt]  [font=\footnotesize]  {$1$};
% Text Node
\draw (146,240.4) node [anchor=east] [inner sep=0.75pt]  [font=\footnotesize]  {$2$};
% Text Node
\draw (233.6,240.4) node [anchor=west] [inner sep=0.75pt]  [font=\footnotesize]  {$3$};
% Text Node
\draw (233.6,160.4) node [anchor=west] [inner sep=0.75pt]  [font=\footnotesize]  {$4$};
% Text Node
\draw (190,203) node [anchor=north] [inner sep=0.75pt]  [font=\footnotesize]  {$5$};
% Text Node
\draw (65.2,33.65) node [anchor=north east] [inner sep=0.75pt]  [font=\footnotesize]  {$1$};
% Text Node
\draw (65.2,106.85) node [anchor=south east] [inner sep=0.75pt]  [font=\footnotesize]  {$2$};
% Text Node
\draw (147.2,106.85) node [anchor=south west] [inner sep=0.75pt]  [font=\footnotesize]  {$3$};
% Text Node
\draw (145.2,33.65) node [anchor=north west][inner sep=0.75pt]  [font=\footnotesize]  {$4$};
% Text Node
\draw (112.98,69.16) node [anchor=west] [inner sep=0.75pt]  [font=\footnotesize]  {$5$};
% Text Node
\draw (360,157) node [anchor=south] [inner sep=0.75pt]  [font=\footnotesize]  {$1$};
% Text Node
\draw (310.26,207.2) node [anchor=south east] [inner sep=0.75pt]  [font=\footnotesize]  {$2$};
% Text Node
\draw (360,237) node [anchor=south] [inner sep=0.75pt]  [font=\footnotesize]  {$3$};
% Text Node
\draw (410.26,207.33) node [anchor=south west] [inner sep=0.75pt]  [font=\footnotesize]  {$4$};
% Text Node
\draw (360,204) node [anchor=north] [inner sep=0.75pt]  [font=\footnotesize]  {$5$};
% Text Node
\draw (95.71,125.14) node [anchor=north west][inner sep=0.75pt]   [align=left] {5A};
% Text Node
\draw (260.57,124.71) node [anchor=north west][inner sep=0.75pt]   [align=left] {5B};
% Text Node
\draw (430.57,124.57) node [anchor=north west][inner sep=0.75pt]   [align=left] {5C};
% Text Node
\draw (180.57,254.86) node [anchor=north west][inner sep=0.75pt]   [align=left] {5D};
% Text Node
\draw (350.29,255.11) node [anchor=north west][inner sep=0.75pt]   [align=left] {5E};

\end{tikzpicture}
    \caption{Graphs representing all connected, loopless five-crossing link diagrams}
    \label{fig:five_crossing_four_regular_graphs}
\end{figure}

Notice that 5A and 5B each contain a tripled edge, meaning each has a component eliminable via an RII-move (see the proof of Lemma \ref{lemma:four_crossing_one_RIII}, above).

A link representing 5C, if alternating, is a non-trivial knot. If adjacent crossings share an overstrand, an RII-move reduces the diagram to one with three crossings, and we apply Lemma \ref{lemma:four_crossing_one_RIII}.

For graph 5D, there are three possible situations for non-alternating diagrams. First, it may be that either the upper or lower pair of crossings shares an under- or overcrossing. In this case, an RII-move will reduce to a three-crossing diagram. Otherwise, an overstrand will be shared between the center crossing and at least one other. These situations are modeled, up to planar isotopy and possibly reflection, in Figure \ref{fig:link_projections_with_5D}. In the first case, the diagram is not an unlink. In the second, a single RIII-move followed by an RI-move reduces to a diagram with four crossings.

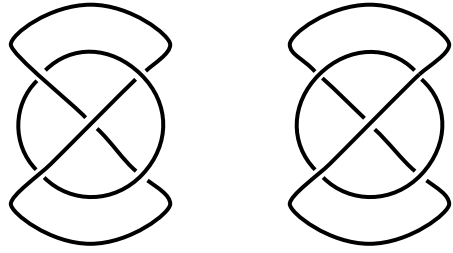
\begin{figure}[h!]
    \centering

\tikzset{every picture/.style={line width=0.75pt}} %set default line width to 0.75pt        

\begin{tikzpicture}[x=0.75pt,y=0.75pt,yscale=-1,xscale=1]
%uncomment if require: \path (0,174); %set diagram left start at 0, and has height of 174

%Curve Lines [id:da23793773506284477] 
\draw [line width=1.5]    (102.64,57.45) .. controls (97,62.73) and (65.93,94.2) .. (60,100) .. controls (54.07,105.8) and (40.07,115.5) .. (40,120) .. controls (39.93,124.5) and (59.82,140.07) .. (80,140) .. controls (100.18,139.93) and (120.07,124.5) .. (120,120) .. controls (119.93,115.5) and (113.93,112.07) .. (108.6,108.2) ;
%Curve Lines [id:da3990828799263516] 
\draw [line width=1.5]    (77.55,76.73) .. controls (72.05,70.37) and (40.07,45.21) .. (40,40) .. controls (39.93,34.79) and (59.64,20.07) .. (80,20) .. controls (100.36,19.93) and (119.93,35.36) .. (120,40) .. controls (120.07,44.64) and (112.09,49.27) .. (107.73,53.27) ;
%Curve Lines [id:da8847823244445251] 
\draw [line width=1.5]    (53.1,58.2) .. controls (40.9,70) and (40.5,90.4) .. (52.7,102.8) ;
%Curve Lines [id:da6647945293134071] 
\draw [line width=1.5]    (56.73,106.73) .. controls (70.07,120.07) and (92.07,120.07) .. (106.07,105.13) .. controls (120.07,90.2) and (120.16,69.48) .. (105.18,54.91) .. controls (90.2,40.33) and (70.45,40) .. (57.36,52.91) ;
%Curve Lines [id:da7224865318482638] 
\draw [line width=1.5]    (83.27,82.2) .. controls (90,88.2) and (96.6,98.47) .. (102.87,103.53) ;
%Curve Lines [id:da300071551082184] 
\draw [line width=1.5]    (245,55) .. controls (239.93,59.53) and (205.93,94.2) .. (200,100) .. controls (194.07,105.8) and (180.07,115.5) .. (180,120) .. controls (179.93,124.5) and (199.82,140.07) .. (220,140) .. controls (240.18,139.93) and (260.07,124.5) .. (260,120) .. controls (259.93,115.5) and (253.93,112.07) .. (248.6,108.2) ;
%Curve Lines [id:da26633327563022147] 
\draw [line width=1.5]    (192.64,53.57) .. controls (187.14,47.21) and (180.07,45.21) .. (180,40) .. controls (179.93,34.79) and (199.64,20.07) .. (220,20) .. controls (240.36,19.93) and (259.93,35.36) .. (260,40) .. controls (260.07,44.64) and (250.07,50.47) .. (245,55) ;
%Curve Lines [id:da1842220015064412] 
\draw [line width=1.5]    (242.64,53) .. controls (230.21,40.14) and (209.5,40.43) .. (194.79,55) .. controls (180.07,69.57) and (180.09,89.91) .. (192.7,102.8) ;
%Curve Lines [id:da9793880664480952] 
\draw [line width=1.5]    (196.73,106.73) .. controls (210.07,120.07) and (232.07,120.07) .. (246.07,105.13) .. controls (260.07,90.2) and (260.07,69.71) .. (246.11,57.44) ;
%Curve Lines [id:da4844022586796852] 
\draw [line width=1.5]    (222.64,82.57) .. controls (229.38,88.57) and (236.6,98.47) .. (242.87,103.53) ;
%Curve Lines [id:da04275897146889274] 
\draw [line width=1.5]    (196.36,56.86) .. controls (203.09,62.86) and (211.5,71.29) .. (217.5,77) ;

\end{tikzpicture}
    \caption{Two five-crossing link diagrams}
    \label{fig:link_projections_with_5D}
\end{figure}

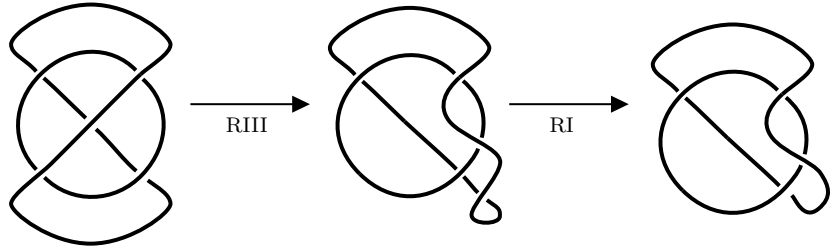
\begin{figure}[h!]
    \centering

\tikzset{every picture/.style={line width=0.75pt}} %set default line width to 0.75pt        

\begin{tikzpicture}[x=0.75pt,y=0.75pt,yscale=-1,xscale=1]
%uncomment if require: \path (0,205); %set diagram left start at 0, and has height of 205

%Curve Lines [id:da8244454841287251] 
\draw [line width=1.5]    (105,65) .. controls (99.93,69.53) and (65.93,104.2) .. (60,110) .. controls (54.07,115.8) and (40.07,125.5) .. (40,130) .. controls (39.93,134.5) and (59.82,150.07) .. (80,150) .. controls (100.18,149.93) and (120.07,134.5) .. (120,130) .. controls (119.93,125.5) and (113.93,122.07) .. (108.6,118.2) ;
%Curve Lines [id:da6139520510974754] 
\draw [line width=1.5]    (52.64,63.57) .. controls (47.14,57.21) and (40.07,55.21) .. (40,50) .. controls (39.93,44.79) and (59.64,30.07) .. (80,30) .. controls (100.36,29.93) and (119.93,45.36) .. (120,50) .. controls (120.07,54.64) and (110.07,60.47) .. (105,65) ;
%Curve Lines [id:da3020473734794742] 
\draw [line width=1.5]    (102.64,63) .. controls (90.22,50.14) and (69.5,50.43) .. (54.79,65) .. controls (40.07,79.57) and (40.09,99.91) .. (52.7,112.8) ;
%Curve Lines [id:da8987277999592778] 
\draw [line width=1.5]    (56.73,116.73) .. controls (70.07,130.07) and (92.07,130.07) .. (106.07,115.13) .. controls (120.07,100.2) and (120.07,79.71) .. (106.11,67.44) ;
%Curve Lines [id:da10133978104714869] 
\draw [line width=1.5]    (82.64,92.57) .. controls (89.38,98.57) and (96.6,108.47) .. (102.87,113.53) ;
%Curve Lines [id:da3186055042065137] 
\draw [line width=1.5]    (56.36,66.86) .. controls (63.09,72.86) and (71.5,81.29) .. (77.5,87) ;
%Curve Lines [id:da46081470067549357] 
\draw [line width=1.5]    (265,66.67) .. controls (257.9,73.1) and (254,80.9) .. (259.9,88.5) .. controls (265.8,96.1) and (285.67,101.95) .. (285.48,109.19) .. controls (285.29,116.43) and (270.6,130.8) .. (271.1,135.8) .. controls (271.6,140.8) and (285,140.5) .. (285.3,136.3) .. controls (285.6,132.1) and (284.13,131.97) .. (278.8,128.1) ;
%Curve Lines [id:da9062157883589919] 
\draw [line width=1.5]    (212.64,65.24) .. controls (207.14,58.88) and (200.07,56.88) .. (200,51.67) .. controls (199.93,46.45) and (219.64,31.74) .. (240,31.67) .. controls (260.36,31.6) and (279.93,47.02) .. (280,51.67) .. controls (280.07,56.31) and (270.07,62.13) .. (265,66.67) ;
%Curve Lines [id:da5814801486296441] 
\draw [line width=1.5]    (262.64,64.67) .. controls (250.21,51.81) and (229.5,52.1) .. (214.79,66.67) .. controls (200.07,81.24) and (200.2,100.4) .. (215,115) ;
%Curve Lines [id:da2703154338914673] 
\draw [line width=1.5]    (215,115) .. controls (230.4,130.1) and (250.2,129.9) .. (265.4,114.7) .. controls (268.9,111.1) and (273.2,106) .. (274.4,102.3) ;
%Curve Lines [id:da8208811300822454] 
\draw [line width=1.5]    (237.5,88.67) .. controls (244.39,95.09) and (258.3,107.26) .. (263.52,112.65) ;
%Curve Lines [id:da08413511636816517] 
\draw [line width=1.5]    (216.36,68.52) .. controls (223.09,74.52) and (231.5,82.95) .. (237.5,88.67) ;
%Curve Lines [id:da9886576411685365] 
\draw [line width=1.5]    (267.12,116.32) .. controls (268.92,118.22) and (273.02,123.32) .. (274.32,124.42) ;
%Curve Lines [id:da21077996322765624] 
\draw [line width=1.5]    (266.8,68.3) .. controls (270.5,71.5) and (280.1,80.4) .. (276.2,96.8) ;
%Curve Lines [id:da4959879573940953] 
\draw [line width=1.5]    (427.07,75) .. controls (419.97,81.43) and (416.07,89.23) .. (421.97,96.83) .. controls (427.87,104.43) and (442.29,108.68) .. (446.31,114.75) .. controls (450.34,120.82) and (452.52,126.64) .. (445.06,132.64) .. controls (437.61,138.64) and (433.56,127.25) .. (431.81,125.75) ;
%Curve Lines [id:da9212147702154535] 
\draw [line width=1.5]    (374.71,73.57) .. controls (369.21,67.21) and (362.14,65.21) .. (362.06,60) .. controls (361.99,54.79) and (381.71,40.07) .. (402.06,40) .. controls (422.42,39.93) and (441.99,55.36) .. (442.06,60) .. controls (442.14,64.64) and (432.13,70.47) .. (427.07,75) ;
%Curve Lines [id:da7896532262881674] 
\draw [line width=1.5]    (424.7,73) .. controls (412.28,60.14) and (391.57,60.43) .. (376.85,75) .. controls (362.14,89.57) and (362.27,108.73) .. (377.07,123.33) ;
%Curve Lines [id:da33544510052462106] 
\draw [line width=1.5]    (377.07,123.33) .. controls (392.47,138.43) and (412.27,138.23) .. (427.47,123.03) .. controls (430.97,119.43) and (435.27,114.33) .. (436.47,110.63) ;
%Curve Lines [id:da8368555650911927] 
\draw [line width=1.5]    (399.57,97) .. controls (406.46,103.42) and (420.37,115.59) .. (425.59,120.99) ;
%Curve Lines [id:da2986138833551033] 
\draw [line width=1.5]    (378.42,76.86) .. controls (385.16,82.86) and (393.57,91.29) .. (399.57,97) ;
%Curve Lines [id:da0015466484043928919] 
\draw [line width=1.5]    (428.87,76.63) .. controls (432.57,79.83) and (442.17,88.73) .. (438.27,105.13) ;
%Straight Lines [id:da6471929393240963] 
\draw    (130,80) -- (187,80) ;
\draw [shift={(190,80)}, rotate = 180] [fill={rgb, 255:red, 0; green, 0; blue, 0 }  ][line width=0.08]  [draw opacity=0] (8.93,-4.29) -- (0,0) -- (8.93,4.29) -- cycle    ;
%Straight Lines [id:da5295311696411817] 
\draw    (290,80) -- (347,80) ;
\draw [shift={(350,80)}, rotate = 180] [fill={rgb, 255:red, 0; green, 0; blue, 0 }  ][line width=0.08]  [draw opacity=0] (8.93,-4.29) -- (0,0) -- (8.93,4.29) -- cycle    ;

% Text Node
\draw (158.5,90.5) node  [font=\footnotesize] [align=left] {\begin{minipage}[lt]{16.08pt}\setlength\topsep{0pt}
RIII
\end{minipage}};
% Text Node
\draw (318.5,90.5) node  [font=\footnotesize] [align=left] {\begin{minipage}[lt]{12.53pt}\setlength\topsep{0pt}
RI
\end{minipage}};

\end{tikzpicture}
    \caption{Two Reidemeister moves eliminating a crossing}
    \label{fig:R3_moves_for_5D}
\end{figure}

If two adjacent crossings in 5E from among crossings 2 through 4 share an overstrand, then an RII-move reduces to a three-crossing diagram. Likewise if crossings 1 and 5 share an overstrand. Suppose then that crossings 1 and 5 do not share an overstrand, likewise for any given pair among crossings 2 through 4, and that crossing 1 shares an overstrand with either crossing 2 or 4. Equivalently, crossing 5 shares an overstrand with the other of 2 or 4. Then, up to planar isotopy and possible reflection, the situation matches that in Figure \ref{fig:five_crossing_figure_eight}, and the diagram represents a figure-eight knot.

\end{proof}

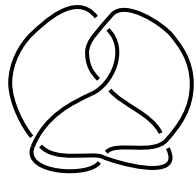
\begin{figure}[h!]
    \centering

\tikzset{every picture/.style={line width=0.75pt}} %set default line width to 0.75pt        

\begin{tikzpicture}[x=0.75pt,y=0.75pt,yscale=-1,xscale=1]

%Curve Lines [id:da04033137965341371] 
\draw [line width=1.5]    (60,100.08) .. controls (54.46,113.77) and (88.46,115.46) .. (93.38,107.92) ;
%Curve Lines [id:da13567418628460282] 
\draw [line width=1.5]    (89.5,73.5) .. controls (99.94,68.94) and (107.92,51.09) .. (97.7,41.2) ;
%Curve Lines [id:da5987881057128168] 
\draw [line width=1.5]    (98.73,71.57) .. controls (104.73,78.97) and (119.47,83.41) .. (124.21,93.21) ;
%Curve Lines [id:da9377246256904288] 
\draw [line width=1.5]    (93.01,66.23) .. controls (89.15,62.43) and (87.27,58.5) .. (87.27,52.68) .. controls (87.27,46.86) and (92.86,41.91) .. (99.61,33.93) .. controls (106.36,25.95) and (125.38,38.71) .. (129.61,43.93) .. controls (133.83,49.15) and (139.81,56.63) .. (139.81,68.63) .. controls (139.81,80.63) and (133.93,89.15) .. (127.07,96.79) .. controls (120.21,104.42) and (101.74,97.12) .. (96.93,102.85) ;
%Curve Lines [id:da690535885053271] 
\draw [line width=1.5]    (59.5,95.07) .. controls (55.28,90.07) and (48.7,78) .. (48.7,68) .. controls (48.7,58) and (54.17,48.5) .. (59.5,43.5) .. controls (64.83,38.5) and (81.7,21.6) .. (91.9,35) ;
%Curve Lines [id:da8569068441107841] 
\draw [line width=1.5]    (60,100.08) .. controls (66.33,84.42) and (79.72,78.06) .. (89.5,73.5) ;
%Curve Lines [id:da8120027992124768] 
\draw [line width=1.5]    (63.93,99.64) .. controls (71.26,108.09) and (89.79,102.64) .. (95.04,105.14) .. controls (100.3,107.64) and (136.5,118.5) .. (127.79,100.79) ;

\end{tikzpicture}
    \caption{A diagram of the figure-eight knot}
    \label{fig:five_crossing_figure_eight}
\end{figure}

We are now ready to prove our main result. We recall that a bridge trisection of $K$ is determined by a triple of unlinks in bridge position \cite{mz}, and thanks to Remark 3.6 of \cite{jmmz}, we know that triple points correspond exactly to RIII-moves in untyings of these unlinks. Leveraging lemmas \ref{lemma:four_crossing_one_RIII} and \ref{lemma:fivecross}, we can bound the number of such moves. 

\begin{thm}\label{thm:five-crossing-ribbon}

Every 2-knot embedded in $S^4$ admitting a bridge trisection with at most five crossings is ribbon.

\end{thm}

\begin{proof}

If a bridge-trisected sphere $K \subset S^4$ has a tri-plane diagram $(T_\a, T_\b, T_\g)$ that has four or fewer total crossings, then every unlink diagram arising in a pairwise union of the tangles in the diagram will have at most four crossings, and so it will admit an untying with no RIII-moves. In this case, $K$ is ribbon.

By a result of Satoh \cite{S05}, we know that a 2-knot admitting an immersion with at most three triple points admits an immersion with none. So it will suffice to show that a five-crossing bridge trisection has at most three triple points.

Suppose now that $K$ has five total crossings, and suppose that none of the three tangles contains all five crossings; that is, that at least two tangles have at least one crossing each. Then at most one unlink pair, suppose without loss of generality $T_\a \cup T_\b$, has five crossings. Since all unlink pairs with fewer than five crossings can be untied without RIII-moves, $K$ will have at most one triple point, arising in the disk system capping off the $T_\a \cup T_\b$ unlink.

If instead all crossings are concentrated into a single tangle (say $T_\a$), then both pairs in which that tangle appears will have five crossings, but the third unlink pair will have none. Thus, $K$ will have at most two triple points, corresponding to any RIII-moves used in an untying of either of the five-crossing unlinks.

We now see that for a tri-plane diagram with five or fewer crossings, no more than two RIII-moves are required to untie all three unlink diagrams obtained by pairwise unions of tangles. So a knotted surface admitting such a diagram has at most two triple points, and by Satoh's result \cite{S05}, it can be constructed with none.
\end{proof}

This result has two immediate corollaries, the first restricting the possibilities for low-crossing surface knots to those of ribbon type.

\begin{cor}
    Knotted surfaces admitting tri-plane diagrams with five or fewer total crossings are ribbon.
\end{cor}

Further, this gives us the first known crossing number for a non-trivial knotted surface.

\begin{cor}\label{cor:2-twist}
    The 2-twist spin of the trefoil has crossing number 6.
\end{cor}

\begin{proof}
    Since we know that the 2-twist spun trefoil has triple-point number four \cite{S03}, all that remains is to exhibit a six-crossing diagram of the knot. Figure \ref{fig:tri-plane} is just such a diagram, due to Maxon, Meier, and Rollins \cite{diagram}. For a detailed discussion of this diagram, see Section 7.1 of \cite{rainbow}.
\end{proof}

This is the first known crossing number for a non-trivial knotted surface, and this example illustrates the difficulty in computing them. Prior to the discovery of this diagram, the lowest-crossing known diagram of the knot had fourteen crossings, appearing in \cite{wolf}. The only other 2-knot with a known six-crossing diagram is the spun trefoil (Figure \ref{fig:spuntref}). We believe this to be minimal.

\begin{figure}[h!]
    \centering

\tikzset{every picture/.style={line width=0.75pt}} %set default line width to 0.75pt        

\begin{tikzpicture}[x=0.75pt,y=0.75pt,yscale=-1,xscale=1]
%uncomment if require: \path (0,115); %set diagram left start at 0, and has height of 115

%Straight Lines [id:da024763371435511883] 
\draw    (20,80) -- (110,80) ;
%Curve Lines [id:da9621657479798924] 
\draw [color={rgb, 255:red, 255; green, 0; blue, 0 }  ,draw opacity=1 ]   (30,80) .. controls (30.13,70.25) and (40.13,70.25) .. (40,80) ;
%Curve Lines [id:da17161918019161615] 
\draw [color={rgb, 255:red, 255; green, 0; blue, 0 }  ,draw opacity=1 ]   (70,80) .. controls (70.13,70.25) and (80.13,70.25) .. (80,80) ;
%Curve Lines [id:da08504069022064498] 
\draw [color={rgb, 255:red, 255; green, 0; blue, 0 }  ,draw opacity=1 ]   (50,80) .. controls (50.13,40) and (100.13,40) .. (100,80) ;
%Curve Lines [id:da7974063803804426] 
\draw [color={rgb, 255:red, 255; green, 0; blue, 0 }  ,draw opacity=1 ]   (60,80) .. controls (60,56) and (90.5,55.5) .. (90,80) ;
%Straight Lines [id:da234691155502616] 
\draw    (130,80) -- (220,80) ;
%Curve Lines [id:da3099835388393972] 
\draw [color={rgb, 255:red, 0; green, 0; blue, 255 }  ,draw opacity=1 ]   (140,80) .. controls (140.13,70.25) and (150.13,70.25) .. (150,80) ;
%Curve Lines [id:da9257399587420313] 
\draw [color={rgb, 255:red, 0; green, 0; blue, 255 }  ,draw opacity=1 ]   (170,80) .. controls (170.13,70.25) and (180.13,70.25) .. (180,80) ;
%Curve Lines [id:da5926071423876684] 
\draw [color={rgb, 255:red, 0; green, 0; blue, 255 }  ,draw opacity=1 ]   (160,80) .. controls (160,56) and (190.5,55.5) .. (190,80) ;
%Curve Lines [id:da8103271750266918] 
\draw [color={rgb, 255:red, 0; green, 0; blue, 255 }  ,draw opacity=1 ]   (200,80) .. controls (200.13,70.25) and (210.13,70.25) .. (210,80) ;
%Straight Lines [id:da21249860218811645] 
\draw    (240,80) -- (330,80) ;
%Curve Lines [id:da34405677755181463] 
\draw [color={rgb, 255:red, 126; green, 211; blue, 33 }  ,draw opacity=1 ]   (300,80.02) .. controls (300.13,70.27) and (310.13,70.27) .. (310,80.02) ;
%Curve Lines [id:da9964056552023475] 
\draw [color={rgb, 255:red, 126; green, 211; blue, 33 }  ,draw opacity=1 ]   (250,80.02) .. controls (249.83,47.35) and (249.2,13.02) .. (260.75,29.64) ;
%Curve Lines [id:da634264023521474] 
\draw [color={rgb, 255:red, 126; green, 211; blue, 33 }  ,draw opacity=1 ]   (320,80.02) .. controls (319.89,59.39) and (311.76,41.69) .. (298.8,30.09) .. controls (285.84,18.48) and (256,27.39) .. (261.75,34.89) ;
%Curve Lines [id:da28987503464837516] 
\draw [color={rgb, 255:red, 126; green, 211; blue, 33 }  ,draw opacity=1 ]   (261.93,41.88) .. controls (257.75,36.43) and (266.58,35.64) .. (262.8,31.55) ;
%Curve Lines [id:da044799350585625786] 
\draw [color={rgb, 255:red, 126; green, 211; blue, 33 }  ,draw opacity=1 ]   (263,62.52) .. controls (265.4,57.72) and (260.88,59.67) .. (259.5,54.77) .. controls (258.12,49.86) and (259.28,45.81) .. (260.36,44.84) .. controls (261.45,43.86) and (265.64,39.53) .. (263.2,36.89) ;
%Curve Lines [id:da9409280138242854] 
\draw [color={rgb, 255:red, 126; green, 211; blue, 33 }  ,draw opacity=1 ]   (290.67,80.02) .. controls (290.13,38.89) and (268.93,50.09) .. (263.26,43.14) ;
%Curve Lines [id:da05923200131144235] 
\draw [color={rgb, 255:red, 126; green, 211; blue, 33 }  ,draw opacity=1 ]   (280,80.02) .. controls (280.33,40.42) and (264.2,53.22) .. (262.6,57.09) ;
%Curve Lines [id:da43351332866934034] 
\draw [color={rgb, 255:red, 126; green, 211; blue, 33 }  ,draw opacity=1 ]   (263.73,67.82) .. controls (267.02,64.06) and (258.29,63.75) .. (261.63,59.35) ;
%Curve Lines [id:da9149999260885407] 
\draw [color={rgb, 255:red, 126; green, 211; blue, 33 }  ,draw opacity=1 ]   (270,80.02) .. controls (270.13,69.55) and (258.92,69.55) .. (262.25,65.15) ;
%Curve Lines [id:da6985985652416105] 
\draw [color={rgb, 255:red, 126; green, 211; blue, 33 }  ,draw opacity=1 ]   (260,80.27) .. controls (260.4,76.22) and (259.6,72.35) .. (262.27,69.8) ;

% Text Node
\draw (65,90) node  [font=\footnotesize]  {$T_{\alpha }$};
% Text Node
\draw (175,90) node  [font=\footnotesize]  {$T_{\beta }$};
% Text Node
\draw (285,90) node  [font=\footnotesize]  {$T_{\gamma }$};

\end{tikzpicture}
    \caption{A bridge trisection of the spun trefoil}
    \label{fig:spuntref}
\end{figure}
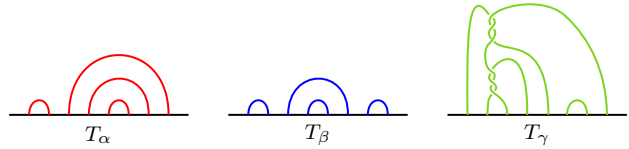

\begin{conj}
The tri-plane crossing number of the spun trefoil is six.
\end{conj}

Finally, we conjecture that the spun trefoil and 2-twist spun trefoil are, in the sense of tri-plane crossing numbers, the simplest 2-knots.

\begin{conj}
If a tri-plane diagram for an embedded $S^2$ in $S^4$ has six or fewer total crossings, then it represents a bridge trisection of an unknotted sphere, the spun trefoil, or the 2-twist spun trefoil.
\end{conj}

\newpage

\printbibliography
    
\end{document}